\documentclass{amsart}

\usepackage{amsmath,amssymb,amscd,amsfonts}
\usepackage{mathtools,color}
\usepackage{graphicx}
\usepackage[hidelinks]{hyperref}
\title[Kuznecov formulas for domains with fractal boundary]
{KUZNECOV FORMULAS FOR DOMAINS WITH FRACTAL BOUNDARY}

\numberwithin{equation}{section}

\newcommand{\vol}{\operatorname{vol}}
\newcommand{\distg}{\operatorname{dist}_g}

\newtheorem{theorem}{Theorem}[section]
\newtheorem{lemma}[theorem]{Lemma}
\newtheorem{proposition}[theorem]{Proposition}
\newtheorem{corollary}[theorem]{Corollary}

\theoremstyle{definition}
\newtheorem{definition}{Definition}[section]

\theoremstyle{remark}
\newtheorem{remark}{Remark}[section]
	\author{Xucheng Hu}
	\address{School of Mathematical Sciences, Zhejiang University, Hangzhou 310027, PR China}
	\email{huxucheng2002@163.com}
   
	\author{Yakun Xi}
	\address{School of Mathematical Sciences, Zhejiang University, Hangzhou 310027, PR China}
	\email{yakunxi@zju.edu.cn}

\begin{document}
\begin{abstract}
Let $(M,g)$ be a compact $d$-dimensional Riemannian manifold without boundary, and let $\{e_j\}_{j=0}^\infty$ be an orthonormal basis of Laplace eigenfunctions with frequencies $\{\lambda_j\}_{j=0}^\infty$. For a domain $\Omega\subset M$, consider

$$
N_\Omega(\lambda)
=
\sum_{\lambda_j\leq\lambda}
\left|\int_\Omega e_j\,dV_g\right|^2.
$$
Suppose that $\partial\Omega$ is $s$-Ahlfors regular, with
$s\in{[d-1,d)}$, and that $\Omega$ satisfies a two-sided
corkscrew condition. We prove that
$$
N_\Omega(\lambda)
=
\vol(\Omega)+O\bigl(\lambda^{s-d}\bigr),
$$
and that this remainder is sharp. We further identify the geometric quantity governing the second term: an exact asymptotic of order $\lambda^{s-d}$ holds precisely when the boundary admits a natural $s$-dimensional crossing volume $\mathcal V^s(\partial\Omega)$, in which case
$$
N_\Omega(\lambda)
=
\vol(\Omega)
-
C_{d,s}\,
\mathcal V^s(\partial\Omega)\,
\lambda^{s-d}
+
o\bigl(\lambda^{s-d}\bigr).
$$ 
For admissible self-similar boundaries, the crossing volume exists in the non-lattice case, and hence the above two-term asymptotic holds. In the lattice case, the second-order behavior is instead governed by a log-periodic profile.
\end{abstract}

\maketitle

\section{Introduction}
\label{sec:introduction}

Let $(M,g)$ be a compact connected ${d}$-dimensional Riemannian manifold
without boundary, and let
$\{e_j\}_{j=0}^{\infty}$ be an orthonormal basis of real eigenfunctions of the
Laplace--Beltrami operator satisfying
$-\Delta_g e_j=\lambda_j^2e_j$, where
$0=\lambda_0\leq\lambda_1\leq\cdots$. Given a finite Borel measure
$\mu$ on $M$, we consider the associated Kuznecov sum
\[N_\mu(\lambda):=\sum_{\lambda_j\leq \lambda}
\left|\int_M e_j\,d\mu\right|^2.\]
A classical problem is to determine the asymptotic behavior of
$N_\mu(\lambda)$ as $\lambda\to\infty$. When $\mu$ is the Riemannian volume
measure induced on an embedded submanifold of $M$,
Zelditch \cite{Zelditch1992} established the classical Kuznecov formula, whose
leading term is given by
\[
  N_{dV_H}(\lambda)
  = (2\pi)^{-({d}-k)}\vol(B^{{d}-k})
    \vol(H)\lambda^{{d}-k}
    + O\bigl(\lambda^{{d}-k-1}\bigr),
\]
where $H\subset M$ is an embedded submanifold of dimension $k$ and $dV_H$
is its induced Riemannian volume measure. More recently, Wyman and
Xi \cite{WymanXi2023} refined this result by deriving a two-term Kuznecov sum
formula and identifying the next-order term.

It is natural to ask what remains true when the smoothness assumption on
$\mu$ is dropped. In this direction, Xi \cite{XiFractalKuznecov} studied
Kuznecov formulae for fractal measures. In particular, for an
$s$-Ahlfors regular measure satisfying an appropriate averaged density
condition, a one-term asymptotic was obtained, with leading order
$\lambda^{{d}-s}$. More precisely,
\[
  N_\mu(\lambda)
  = (2\pi)^{-({d}-s)}\vol(B^{{d}-s})
    A_\mu\lambda^{{d}-s}
    + o\bigl(\lambda^{{d}-s}\bigr),
  \qquad \lambda\to\infty,
\]
where $A_\mu$ is the averaged $s$-density constant of $\mu$, which may
be regarded as an effective $s$-dimensional volume of the measure. This result extends the
classical theory from smooth submanifold measures to a broad class of singular
and fractal measures, including, for example,
volume measures on sufficiently regular subsets of smooth submanifolds
with nonsmooth or fractal boundary.

The problem considered here is of a different nature. We take
$\mu=\mathbf 1_\Omega\,dV_g$, where $\Omega\subset M$ is a domain, and write
\[
N_\Omega(\lambda)
:=
\sum_{\lambda_j\leq\lambda}
\left|\int_\Omega e_j\,dV_g\right|^2.
\]
There is a natural physical interpretation of this quantity. If a drum of shape
$M$ is excited uniformly on the region $\Omega$, then
$\int_\Omega e_j\,dV_g$ is the amplitude of the $j$th eigenmode produced
by this excitation, and $N_\Omega(\lambda)$ measures the cumulative squared
response of the modes of frequency at most $\lambda$.

In contrast to the singular measures considered above, the measure
$\mathbf 1_\Omega\,dV_g$ is full-dimensional, and Parseval's identity
already determines the leading term:
\[
\sum_{j=0}^{\infty}
\left|\int_\Omega e_j\,dV_g\right|^2
=
\|\mathbf 1_\Omega\|_{L^2(M)}^2
=
\vol(\Omega).
\]
Thus the volume of $\Omega$ is already audible in the leading term, and
the main question is whether finer geometric information about its boundary
can be heard in the rate at which the Kuznecov sum approaches this limit.
Equivalently, we study the decay of the Kuznecov tail
\[
\vol(\Omega)-N_\Omega(\lambda)
=
\sum_{\lambda_j>\lambda}
\left|\int_\Omega e_j\,dV_g\right|^2
\]
as $\lambda\to\infty$. Classical Fourier analysis suggests that this decay
should reflect the singularities of $\mathbf 1_\Omega$, which occur along
$\partial\Omega$. It is therefore natural to expect, in analogy with
Weyl's conjecture (see Remark \ref{rem:Weyl-conj}), that the next-order
behavior is governed by the geometry of $\partial\Omega$. We investigate this question under quantitative geometric hypotheses on
the boundary.

\subsection*{Notation}

We write $d_g$ for the Riemannian distance and $B(x,r)$ for the
corresponding geodesic ball of radius $r$ centered at $x\in M$. The notation $\vol$ denotes the natural
Riemannian or Euclidean volume in the space under consideration. For
subsets of a product of two $d$-dimensional spaces, we write
$\vol^{2d}$ for the corresponding product volume.
We write
$\vol(B^m):=\frac{\pi^{m/2}}{\Gamma(1+m/2)}$ for $m>0$,
which agrees, for integer $m$, with the Euclidean volume of the unit ball in
$\mathbb R^m$. We write $\mathcal H^s$ for the $s$-dimensional Hausdorff
measure associated with $g$.
\subsection{Main results}

To state our results, we first introduce the relevant geometric hypotheses.

\begin{definition}[$s$-Ahlfors regular set]
A nonempty closed set $F\subset M$ is called $s$-Ahlfors regular if there
exist constants $c,C>0$ and $r_0>0$ such that
\[
c r^s\leq \mathcal H^s\bigl(F\cap B(x,r)\bigr)\leq C r^s,
\qquad x\in F,\quad 0<r\leq r_0.
\]
\end{definition}

\begin{definition}[Two-sided corkscrew condition]
A domain $\Omega\subset M$ is said to satisfy the two-sided corkscrew
condition if there exist constants $c_0\in(0,1)$ and $r_0>0$ such that,
for every $x\in\partial\Omega$ and $0<r\leq r_0$, there exist
$y_1\in\Omega$ and $y_2\in M\setminus\overline{\Omega}$ satisfying
\[
B(y_1,c_0r)\subset B(x,r)\cap\Omega,
\qquad
B(y_2,c_0r)\subset
B(x,r)\cap\bigl(M\setminus\overline{\Omega}\bigr).
\]
The first and second requirements are called the interior and exterior
corkscrew conditions, respectively.
\end{definition}

We now state the main results.

\begin{theorem}[Sharp Kuznecov remainder estimate]
\label{thm:kuznecov-two-sided}
Let $\Omega\subset M$ be a domain. Suppose that $\partial\Omega$ is
$s$-Ahlfors regular for some $s\in[d-1,d)$ and that $\Omega$ satisfies
the two-sided corkscrew condition. Then there exist constants $c,C>0$
such that, for all sufficiently large $\lambda$,
\[
c\lambda^{s-d}
\leq
\vol(\Omega)-N_\Omega(\lambda)
\leq
C\lambda^{s-d}.
\]
\end{theorem}

Theorem \ref{thm:kuznecov-two-sided} determines the sharp order of the
Kuznecov tail, but its hypotheses do not in general determine a limiting
coefficient. To characterize when such a coefficient exists, we introduce
the following geometric quantity. For $r>0$, define
\[
E_\Omega(r)
:=
\bigl\{(x,y)\in\Omega\times(M\setminus\Omega):d_g(x,y)<r\bigr\}.
\]

Thus $\vol^{2d}(E_\Omega(r))$ measures, at scale $r$, the volume of pairs lying on opposite sides of the boundary. Under the hypotheses of Theorem \ref{thm:kuznecov-two-sided}, its natural scale is $r^{2d-s}$, as shown in Section \ref{sec:crossing-volume-estimate}. For $s\in{[d-1,d)}$, whenever the following limit exists, we define the
\emph{$s$-dimensional crossing volume} of $\partial\Omega$ by
\begin{equation}
\label{eq:crossing-volume-definition}
\mathcal V^s(\partial\Omega)
:=
{\
\frac{d+1}{\vol(B^{d-1})}
\lim_{r\rightarrow0^+}
r^{s-2d}\vol^{2d}\bigl(E_\Omega(r)\bigr).
}
\end{equation}
{
The normalization is chosen so that when $s=d-1$ the
crossing volume agrees with the usual boundary volume: if
$\partial\Omega$ is smooth, then
\[
\mathcal V^{d-1}(\partial\Omega)=\vol^{d-1}(\partial\Omega).
\]
Thus $\mathcal V^s(\partial\Omega)$ may be viewed as an extension of
boundary volume to higher-dimensional boundaries. It is a geometric
small-scale invariant of the two-sided structure of $\partial\Omega$.
The next theorem shows that this quantity is exactly what is detected
by the second-order Kuznecov asymptotic.

\begin{theorem}[Crossing volume and the Kuznecov second term]
\label{thm:crossing-volume-kuznecov}
Let $\Omega\subset M$ be a domain, and let $s\in[d-1,d)$. Then the
following are equivalent:
\begin{enumerate}
\item[(i)] The $s$-dimensional crossing volume exists and satisfies
\[
0<\mathcal V^s(\partial\Omega)<\infty.
\]

\item[(ii)] There exists a constant $C_\Omega>0$ such that
\begin{equation}
\label{eq:intro-two-term-kuznecov}
N_\Omega(\lambda)
=
\vol(\Omega)
-
C_\Omega\lambda^{s-d}
+
o\bigl(\lambda^{s-d}\bigr),
\qquad \lambda\to\infty.
\end{equation}
\end{enumerate}
When these equivalent conditions hold,
\begin{equation}
\label{eq:intro-kuznecov-constant}
C_\Omega
=
C_{d,s}\,\mathcal V^s(\partial\Omega),
\end{equation}
where
\[
C_{d,s}
=
\frac{
2^{d-s-1}\Gamma\left(1+d-\frac{s}{2}\right)
}{
\sqrt{\pi}\,
\Gamma\left(\frac{d+3}{2}\right)
\Gamma\left(1-\frac{d-s}{2}\right)
}.
\]
\end{theorem}

An important feature of Theorem \ref{thm:crossing-volume-kuznecov}
is that, unlike Theorem \ref{thm:kuznecov-two-sided}, it requires no
Ahlfors regularity or corkscrew hypothesis. Thus the crossing volume is
not merely a sufficient geometric quantity for producing a second-order
term: its existence is precisely characterized by the corresponding
Kuznecov asymptotic. In this sense, $\mathcal V^s(\partial\Omega)$ is the
natural boundary invariant associated with a Kuznecov tail of order
$\lambda^{s-d}$.

 At the endpoint $s=d-1$, the crossing volume reduces to the usual
codimension-one volume. Recall that a domain $\Omega\subset M$ has finite
perimeter if $\mathbf 1_\Omega\in BV(M)$. In this case its reduced boundary
$\partial^*\Omega$ is $(d-1)$-rectifiable and
$\vol(\partial^*\Omega)$ is its perimeter. This gives the following
consequence.

\begin{corollary}[The finite perimeter case]
\label{cor:codimension-one-kuznecov}
Let $\Omega\subset M$ be a domain of finite perimeter. Then
\[
N_\Omega(\lambda)
=
\vol(\Omega)
-\frac{1}{\pi}\vol(\partial^*\Omega)\lambda^{-1}
+o(\lambda^{-1}),
\qquad \lambda\to\infty,
\]
where $\partial^*\Omega$ denotes the reduced boundary. In particular, if
$\Omega$ is a Lipschitz domain, then
\[
N_\Omega(\lambda)
=
\vol(\Omega)
-\frac{1}{\pi}\vol(\partial\Omega)\lambda^{-1}
+o(\lambda^{-1}).
\]
\end{corollary}
In the Euclidean setting, the corresponding finite-perimeter Fourier-tail asymptotic was obtained by Beretti and Gennaioli \cite{BerettiGennaioli2024}. Corollary \ref{cor:codimension-one-kuznecov} provides the corresponding Laplace-spectral analog on compact Riemannian manifolds.
The geometric identity underlying
Corollary \ref{cor:codimension-one-kuznecov}, namely the identification of
$\mathcal V^{d-1}(\partial\Omega)$ with the codimension-one volume of the
reduced boundary, is proved in
Proposition \ref{prop:codimension-one-crossing-volume}. 
This codimension-one formula makes the analogy with the second term in
Weyl's law especially transparent.}

\begin{remark}[Weyl asymptotics and the Berry conjecture]
\label{rem:Weyl-conj}
For a smooth bounded domain $\Omega\subset\mathbb R^d$, let
\[
N^{D/N}(\lambda):=\#\{j:\lambda_j^{D/N}\leq\lambda\}
\]
denote the counting function for the Dirichlet or Neumann Laplacian.
Weyl's two-term conjecture predicts
\[
\begin{aligned}
N^{D/N}(\lambda)
&=
(2\pi)^{-d}\vol(B^d)\vol(\Omega)\lambda^d \\
&\quad
\mp \frac14(2\pi)^{-(d-1)}
\vol(B^{d-1})\vol(\partial\Omega)\lambda^{d-1}
+o(\lambda^{d-1}),
\end{aligned}
\]
where the minus and plus signs correspond to the Dirichlet and Neumann
boundary conditions. Thus the first correction to the leading volume term
detects codimension-one boundary volume. Ivrii \cite{Ivrii1980} proved
this conjecture for sufficiently smooth domains under the assumption that
the set of periodic billiard trajectories has measure zero. More recently, Frank and
Larson \cite{FrankLarson2025} established two-term asymptotics for Riesz
means on bounded Lipschitz domains.

For fractal boundaries, the analogous modification of the second Weyl term
is related to Berry's conjecture. Levitin and
Vassiliev \cite{LevitinVassiliev1996} studied self-similar fractal domains
by renewal-theoretic methods and showed that the boundary dimension
determines the scale of the correction, while its finer behavior depends
on the arithmetic structure of the defining similarities. In particular,
constant and log-periodic next-order terms can arise. Their spectral
problem is different from ours: they study the Laplacian on the fractal
domain itself, whereas here the eigenfunctions belong to the ambient
manifold.

Theorems \ref{thm:kuznecov-two-sided}
and \ref{thm:crossing-volume-kuznecov} exhibit the corresponding phenomenon
for Kuznecov sums. For an $s$-dimensional boundary, the natural scale is
$\lambda^{s-d}$, and an exact term at this scale is equivalent to the
existence of the crossing volume $\mathcal V^s(\partial\Omega)$. At
$s=d-1$ this reduces to the usual boundary volume, while for $s>d-1$ it
provides a higher-dimensional boundary invariant. For the self-similar
boundaries considered below, the non-lattice case yields a genuine second
term, whereas the lattice case exhibits log-periodic behavior.
\end{remark}

\begin{remark}[Relation to heat-content asymptotics]
\label{rem:heat-content}
The heat-kernel quantity underlying
Theorem \ref{thm:crossing-volume-kuznecov} is closely related to the
classical notion of heat loss. In the Euclidean setting, van den Berg and
Gittins \cite{vanDenBergGittins2015} considered the free heat evolution
with initial temperature $\mathbf 1_\Omega$ and studied the heat loss
\[
F_\Omega(t)
=
\vol(\Omega)-\int_\Omega e^{t\Delta}\mathbf 1_\Omega\,dx
=
\int_\Omega\int_{\mathbb R^d\setminus\Omega}
p_t(x,y)\,dy\,dx.
\]
They obtained two-sided estimates for $F_\Omega(t)$ in terms of quantities
of the form
\[
\int_\Omega
\vol\bigl(B(x,c\sqrt t)\setminus\Omega\bigr)\,dx
=
\vol^{2d}\bigl(E_\Omega(c\sqrt t)\bigr),
\]
which measure the volume of pairs crossing the boundary at scale
$\sqrt t$.

Heat asymptotics for domains with fractal boundary have also been studied
in several related settings. Van den Berg \cite{vanDenBerg1994} investigated
heat content and Brownian motion for regions with fractal boundary, while
van den Berg and den Hollander \cite{vanDenBergDenHollander1999} obtained
log-periodic terms for planar domains with fractal polygonal boundary.
More recently, Park and Xiao \cite{ParkXiao2023} exhibited an
arithmetic/non-arithmetic dichotomy for spectral heat content on
self-similar fractal domains using renewal theory.
Rozanova-Pierrat, Teplyaev, Winter, and
Z\"ahle \cite{RozanovaPierratTeplyaevWinterZaehle2025} studied short-time
heat content for domains with smooth or fractal boundary and related its
asymptotic behavior to Minkowski sausages and fractal-curvature quantities.

The heat-kernel quantity arising in the present Kuznecov problem is closely
related to these heat-content problems, but the geometric quantity used here
is different. Whereas Minkowski-sausage asymptotics measure the volume of a
one-sided boundary layer, the crossing volume is determined by the two-point
crossing quantity $\vol^{2d}(E_\Omega(r))$, which measures pairs lying on
opposite sides of the boundary at scale $r$. This crossing-pair geometry is
linked directly to the ambient heat loss through the heat kernel and, by a
Tauberian argument, to the unsmoothed Kuznecov sum. In particular,
Theorem \ref{thm:crossing-volume-kuznecov} shows that the existence of an
exact Kuznecov second term is equivalent to the existence of the
corresponding crossing volume.
\end{remark}

It remains to identify geometric classes for which the crossing-volume
hypothesis of Theorem \ref{thm:crossing-volume-kuznecov} holds. We now
describe a class of self-similar fractal boundaries for which this can be
done.

\subsection{Admissible self-similar boundaries}

We call a boundary \emph{admissible self-similar} if it is assembled from
finitely many self-similar pieces in $\mathbb R^d$ satisfying the open set
condition and suitable compatibility conditions.
These conditions ensure that the two local sides fit together and that
interactions between distinct pieces are lower-dimensional. This class
includes familiar examples such as the Koch snowflake and the
$(n,c)$-von Koch snowflakes studied by Keleti and
Paquette \cite[Section 1 and Theorem 3.1]{KeletiPaquette2010}.

The lattice type is inherited from the generating iterated function
systems (IFSs). An admissible self-similar boundary is called lattice or
non-lattice according to the arithmetic type of its local generating IFSs.
The precise definitions are given in
Section \ref{sec:self-similar-crossing-volumes}. The self-similar geometry
determines whether the normalized crossing volume converges or exhibits
log-periodic behavior. For the spectral statements below, after a translation
and dilation we regard these Euclidean domains as domains in the flat torus
$\mathbb T^d$, without changing their contraction ratios or lattice type.

\begin{theorem}[Kuznecov asymptotic in the non-lattice case]
\label{thm:admissible-self-similar-kuznecov}
Let $\Omega\subset\mathbb T^d$ be a domain whose boundary is a
non-lattice admissible self-similar boundary of dimension $s\in(d-1,d)$.
Then
\[
0<\mathcal V^s(\partial\Omega)<\infty,
\]
and \begin{equation*}
N_\Omega(\lambda)
=
\vol(\Omega)
-
C_{d,s}\,\mathcal V^s(\partial\Omega)\,\lambda^{s-d}
+
o\bigl(\lambda^{s-d}\bigr),
\qquad \lambda\to\infty.
\end{equation*}
\end{theorem}

\begin{remark}
The self-avoiding $(n,c)$-von Koch snowflakes studied by Keleti and
Paquette \cite[Section 1 and Theorem 3.1]{KeletiPaquette2010} form a
family of admissible self-similar boundaries in the sense of this
paper. In particular, the $(3,\tfrac14)$-von Koch snowflake is
non-lattice. See Remark \ref{rem:nc-snowflake-admissibility} for a
detailed verification.
\end{remark}

In the lattice case, self-similar scaling generally produces log-periodic
oscillations rather than a constant crossing-volume limit. The classical Koch snowflake is a basic lattice example.
Accordingly, one cannot in general expect the constant second-order
asymptotic of Theorem \ref{thm:crossing-volume-kuznecov}. Instead, we obtain
the following log-periodic bounds.

\begin{theorem}[Kuznecov bounds in the lattice case]
\label{thm:lattice-kuznecov}
Let $\Omega\subset\mathbb T^d$ be a domain whose boundary is a lattice
admissible self-similar boundary of dimension $s\in(d-1,d)$. Suppose that
the maximal lattice spans of its local pieces are commensurable. Then there
exist $q>1$, a positive, right-continuous, multiplicatively $q$-periodic
function $P_\Omega:(0,\infty)\to(0,\infty)$, bounded above and below by
positive constants, and a function
$\varepsilon:(0,\infty)\to(0,1)$ with
$\varepsilon(\lambda)\to0$ as $\lambda\to\infty$, such that, as
$\lambda\to\infty$,
\begin{align}
\label{eq:lattice-kuznecov-estimate}
\vol(\Omega)
&-\lambda^{s-d}
P_\Omega\bigl((1-\varepsilon(\lambda))\lambda\bigr)
+o(\lambda^{s-d})
\notag\\
&\leq N_\Omega(\lambda)
\leq
\vol(\Omega)
-\lambda^{s-d}
P_\Omega\bigl((1+\varepsilon(\lambda))\lambda\bigr)
+o(\lambda^{s-d}).
\end{align}
The two little-oh terms may be different.
\end{theorem}

Theorem \ref{thm:lattice-kuznecov} shows that the lattice case retains the
same natural scale $\lambda^{s-d}$, but the coefficient is governed by a multiplicatively periodic
profile. Thus the arithmetic distinction between lattice and non-lattice
self-similarity is reflected spectrally: the non-lattice case yields a
genuine two-term asymptotic, while the lattice case exhibits log-periodic
oscillation. Because the periodic profile $P_\Omega$ need not be continuous,
the spectral conclusion is naturally expressed by the two-sided bounds
above rather than by a single log-periodic asymptotic formula.

\subsection{Organization of the paper}

The remainder of the paper is organized as follows. In
Section \ref{sec:crossing-volume-estimate}, we
establish the two-sided estimates for $\vol^{2d}(E_\Omega(r))$ underlying
Theorem \ref{thm:kuznecov-two-sided}.
In Section \ref{sec:heat-tauberian}, we combine the classical small-time
asymptotics of the heat kernel with a Tauberian argument to prove
Theorem \ref{thm:crossing-volume-kuznecov}.
Section \ref{sec:self-similar-crossing-volumes} introduces the precise
self-similar framework and establishes the corresponding crossing-volume
asymptotics. In particular, it proves
Theorem \ref{thm:admissible-self-similar-kuznecov}. Finally,
Section \ref{sec:lattice-case} derives the spectral consequences in the
lattice case and proves
Theorem \ref{thm:lattice-kuznecov}.

\subsection*{Declaration on the use of generative AI}
ChatGPT (OpenAI, GPT-5.5/5.6) was used for language polishing and 
proofreading of the
manuscript. The authors take full responsibility for the mathematical
content and the final text.

\subsection*{Acknowledgment}  This project is supported by the National Key Research and Development Program of China No. 2022YFA1007200,  Zhejiang Provincial Natural Science Foundation of China under
Grant No. LR25A010001, and NSFC Grant No.
12571107. 

\section{Geometric estimates for the Kuznecov remainder}
\label{sec:crossing-volume-estimate}

In this section, we establish the geometric estimate underlying
Theorem \ref{thm:kuznecov-two-sided}. Under the hypotheses of that theorem,
we prove the two-sided bound
\[
  c r^{2d-s}
  \leq \vol^{2d}\bigl(E_\Omega(r)\bigr)
  \leq C r^{2d-s}
\]
for all sufficiently small $r>0$. Combined with the small-time asymptotics
of the heat kernel and a Tauberian theorem, this estimate yields
Theorem \ref{thm:kuznecov-two-sided}. Since this analytic step is essentially
the same as the one used to prove
Theorem \ref{thm:crossing-volume-kuznecov}, we treat the two arguments
together in Section \ref{sec:heat-tauberian}. We therefore focus here on the
following geometric estimate.

\begin{proposition}
\label{prop:crossing-volume}
Let $(M,g)$ be a compact $d$-dimensional Riemannian manifold without
boundary, and let $\Omega\subset M$ be a domain. Suppose that
$\partial\Omega$ is $s$-Ahlfors regular for some $s\in[d-1,d)$ and that
$\Omega$ satisfies the two-sided corkscrew condition. Then there exist
constants $c,C>0$ and $r_0>0$ such that
\[
  c r^{2d-s}
  \leq \vol^{2d}\bigl(E_\Omega(r)\bigr)
  \leq C r^{2d-s},
  \qquad 0<r\leq r_0.
\]
\end{proposition}

The proof uses the following standard uniform volume estimate for small
geodesic balls, which follows from the compactness of $M$.

\begin{lemma}[Uniform volume bounds for small geodesic balls]
\label{lem:uniform-ball-volume}
There exist constants $c_M,C_M>0$ and $r_M>0$ such that
\[
  c_M r^{d}
  \leq \vol\bigl(B(x,r)\bigr)
  \leq C_M r^{d},
  \qquad x\in M,\quad 0<r\leq r_M.
\]
\end{lemma}

\begin{proof}[Proof of Proposition \ref{prop:crossing-volume}]
Set $F:=\partial\Omega$. By Ahlfors regularity, there exist constants
$c_A,C_A>0$ and $r_A>0$ such that
\begin{equation}
\label{eq:ahlfors-section2}
  c_A\rho^s
  \leq \mathcal H^s\bigl(F\cap B(z,\rho)\bigr)
  \leq C_A\rho^s,
  \qquad z\in F,\quad 0<\rho\leq r_A.
\end{equation}
Since $F$ is compact, a finite covering by balls of radius at most
$r_A$, together with \eqref{eq:ahlfors-section2}, shows that
$0<\mathcal H^s(F)<\infty$.
Let $c_0\in(0,1)$ and $r_{\mathrm{cs}}>0$ be the constants in the
two-sided corkscrew condition. We may take
$r_*:=\min\{r_A,r_{\mathrm{cs}},r_M/2\}$.
All estimates below are understood for $0<r\leq r_*$.

We first prove the upper bound. Write
$F_r:=\{x\in M:\distg(x,F)<r\}$.
If $x\in\Omega$, $y\in M\setminus\Omega$, and $d_g(x,y)<r$, then a
minimizing geodesic from $x$ to $y$ meets $F$. Consequently
$x\in F_r$, and Fubini's theorem together with
Lemma \ref{lem:uniform-ball-volume} gives
\begin{equation}
\label{eq:E-upper-first}
\begin{aligned}
  \vol^{2d}\bigl(E_\Omega(r)\bigr)
  &\leq
    \int_{\Omega\cap F_r}
      \vol\bigl(B(x,r)\bigr)\,dV_g(x) \\
  &\leq C_Mr^{d}\vol(F_r).
\end{aligned}
\end{equation}

To complete the upper bound, we estimate $\vol(F_r)$.

Let $\{z_j\}_{j=1}^{N}\subset F$ be a finite maximal $r$-separated set,
which in particular means that the balls $B(z_j,r/2)$ are pairwise
disjoint and that $F\subset\bigcup_{j=1}^{N}B(z_j,r)$.
Such a set exists because $F$ is compact.
The lower Ahlfors bound in \eqref{eq:ahlfors-section2} then yields
\[
  N c_A(r/2)^s
  \leq
  \sum_{j=1}^{N}
  \mathcal H^s\bigl(F\cap B(z_j,r/2)\bigr)
  \leq \mathcal H^s(F).
\]
Thus $N\leq C r^{-s}$. Applying the upper bound in
Lemma \ref{lem:uniform-ball-volume} to the covering of $F_r$, we obtain
\[
  \vol(F_r)
  \leq
  \sum_{j=1}^{N}\vol\bigl(B(z_j,2r)\bigr)
  \leq C N r^{d}
  \leq C r^{{d}-s}.
\]
Substitution into \eqref{eq:E-upper-first} gives
\[
  \vol^{2d}\bigl(E_\Omega(r)\bigr)
  \leq C r^{2{d}-s}.
\]

We next prove the lower bound. Put $\rho=r/4$, and let
$\{z_j\}_{j=1}^{N}\subset F$ be a maximal $2\rho$-separated set. By
maximality, $F\subset\bigcup_{j=1}^{N}B(z_j,2\rho)$.
The upper Ahlfors bound in \eqref{eq:ahlfors-section2} therefore gives
\[
  \mathcal H^s(F)
  \leq
  \sum_{j=1}^{N}
  \mathcal H^s\bigl(F\cap B(z_j,2\rho)\bigr)
  \leq N C_A(2\rho)^s,
\]
and hence $N\geq c\rho^{-s}$.

For each $j$, the two-sided corkscrew condition at $z_j$ and scale $\rho$
provides points $a_j\in\Omega$ and
$b_j\in M\setminus\overline\Omega$ such that
$B(a_j,c_0\rho)\subset\Omega\cap B(z_j,\rho)$ and
$B(b_j,c_0\rho)\subset
(M\setminus\overline\Omega)\cap B(z_j,\rho)$.
If $(x,y)\in B(a_j,c_0\rho)\times B(b_j,c_0\rho)$, then
$x,y\in B(z_j,\rho)$, and hence $d_g(x,y)<2\rho=r/2<r$.
Therefore
$B(a_j,c_0\rho)\times B(b_j,c_0\rho)\subset E_\Omega(r)$.
Since the points $z_j$ are $2\rho$-separated, the balls
$B(z_j,\rho)$ are pairwise disjoint. These product sets are
consequently pairwise disjoint as well. The lower bound in
Lemma \ref{lem:uniform-ball-volume} now implies
\begin{align*}
  \vol^{2d}\bigl(E_\Omega(r)\bigr)
  &\geq
  \sum_{j=1}^{N}
  \vol\bigl(B(a_j,c_0\rho)\bigr)
  \vol\bigl(B(b_j,c_0\rho)\bigr) \\
  &\geq N c_M^2 c_0^{2{d}}\rho^{2{d}}
  \geq c\rho^{2{d}-s}
  =c'r^{2{d}-s}.
\end{align*}
This proves the lower bound and completes the proof.
\end{proof}

\begin{remark}
\label{rem:minkowski-upper-bound}
The proof of the upper bound uses only the finiteness of the upper
$s$-dimensional Minkowski content of the boundary. More precisely, it is
enough that there exist constants $C,r_0>0$ such that
\[
\vol(F_r)\leq Cr^{{d}-s},
\qquad 0<r\leq r_0,
\]
where $F=\partial\Omega$ and $F_r=\{x\in M:\distg(x,F)<r\}$. In
particular, this estimate holds when $F$ is $s$-Ahlfors regular. Moreover,
if $\overline{\dim}_M F<s$, then
\[
  \vol(F_r)=o\bigl(r^{{d}-s}\bigr),
  \qquad r\to0^+.
\]
This observation will be important for the error estimates in
Section \ref{sec:self-similar-crossing-volumes}.
\end{remark}

\section{Crossing volumes and Kuznecov asymptotics}
\label{sec:heat-tauberian}

In this section, we complete the proofs of
Theorems \ref{thm:kuznecov-two-sided}
and \ref{thm:crossing-volume-kuznecov}. We first
derive the small-time behavior of a Gaussian average from 
$\vol^{2d}(E_\Omega(r))$. A local expansion of the heat kernel then transfers this
behavior to a smoothed Kuznecov tail. The desired spectral asymptotics
follow from a monotone Tauberian theorem.

Fix $s\in[d-1,d)$, and set
\[
a_j:=\int_\Omega e_j\,dV_g.
\]
Then Parseval's identity gives
\begin{equation}
\label{eq:counting-and-parseval}
  \lim_{\lambda\to\infty}N_\Omega(\lambda)
  =\sum_{j=0}^{\infty}|a_j|^2
  =\vol(\Omega).
\end{equation}

In this section, it will be convenient to set $\alpha:=2{d}-s$ and
$\beta:=(\alpha-{d})/2=({d}-s)/2$.
In particular, ${0<\beta\leq1/2}$.

\subsection{A Gaussian average}
\label{subsec:gaussian-average}

Recall that
{
$E_\Omega(r)=\{(x,y)\in\Omega\times(M\setminus\Omega):
d_g(x,y)<r\}$, and its volume is denoted by
$\vol^{2d}(E_\Omega(r))$.}
Let $R(x,y):=d_g(x,y)$, and define the finite Borel measure $\nu_\Omega$
on $[0,\operatorname{diam}(M)]$ by
\begin{equation}
\label{eq:distance-pushforward}
  \nu_\Omega
  :=
  R_*
  \bigl(
    (dV_g|_\Omega)\otimes(dV_g|_{M\setminus\Omega})
  \bigr).
\end{equation}

For every $x\in M$ and $r>0$, the geodesic sphere
$\{y\in M:d_g(x,y)=r\}$ has zero Riemannian volume. Indeed, it is
contained in the image under $\exp_x$ of the sphere of radius $r$ in
$T_xM$. The area formula shows that this smooth image has zero
${d}$-dimensional volume. Fubini's theorem therefore gives
$\nu_\Omega(\{r\})=0$ for every $r>0$. Since
$\nu_\Omega(\{0\})=0$ as well, we have
\begin{equation}
\label{eq:E-as-distribution}
  {\vol^{2d}\bigl(E_\Omega(r)\bigr)}
  =\nu_\Omega([0,r))
  =\nu_\Omega([0,r]),
  \qquad r\geq0.
\end{equation}

Define the Gaussian average by
\begin{equation}
\label{eq:gaussian-average}
  G_\Omega(t)
  :=
  \int_\Omega\int_{M\setminus\Omega}
  e^{-d_g(x,y)^2/(4t)}\,dV_g(y)\,dV_g(x),
  \qquad t>0.
\end{equation}
To control the error in the leading heat-kernel coefficient, for $m\geq0$
we also introduce the Gaussian moment
\begin{equation}
\label{eq:gaussian-moment}
  G_{\Omega,m}(t)
  :=
  \int_\Omega\int_{M\setminus\Omega}
  d_g(x,y)^m e^{-d_g(x,y)^2/(4t)}
  \,dV_g(y)\,dV_g(x).
\end{equation}

\begin{lemma}[Gaussian moment bound]
\label{lem:gaussian-moment-bound}
Suppose that {$\vol^{2d}(E_\Omega(r))=O(r^\alpha)$} as
$r\rightarrow0^+$.
Then, for every $m\geq0$,
\begin{equation}
\label{eq:gaussian-moment-bound}
  G_{\Omega,m}(t)
  =
  O\bigl(t^{(\alpha+m)/2}\bigr),
  \qquad t\rightarrow0^+.
\end{equation}
\end{lemma}

{
The proof of this estimate will use the following exact relation between
the Gaussian average and $\vol^{2d}(E_\Omega(r))$.
}

\begin{lemma}
\label{lem:gaussian-average}
{
For every $t>0$,
\begin{equation}
\label{eq:gaussian-stieltjes}
  G_\Omega(t)
  =
  \frac{1}{2t}
  \int_0^\infty
  r e^{-r^2/(4t)}\vol^{2d}\bigl(E_\Omega(r)\bigr)\,dr.
\end{equation}
}
\end{lemma}

\begin{proof}
{
By \eqref{eq:distance-pushforward},
$G_\Omega(t)=\int_{[0,\infty)}e^{-r^2/(4t)}\,d\nu_\Omega(r)$.
Equation \eqref{eq:E-as-distribution} and Stieltjes integration by parts
give \eqref{eq:gaussian-stieltjes}, since
$\vol^{2d}(E_\Omega(0))=0$, the function
$r\mapsto\vol^{2d}(E_\Omega(r))$ is bounded, and
$e^{-r^2/(4t)}\to0$ as $r\to\infty$.
}
\end{proof}

\begin{proof}[Proof of Lemma \ref{lem:gaussian-moment-bound}]
{
Choose $r_0>0$ and $C_0>0$ such that
$\vol^{2d}(E_\Omega(r))\leq C_0r^\alpha$ for $0<r\leq r_0$, and set
$B_\Omega:=\vol^{2d}\bigl(\Omega\times(M\setminus\Omega)\bigr)
=\vol(\Omega)\vol(M\setminus\Omega)$.
For $m\geq0$, let
\[
  C_m:=\sup_{u\geq0}u^m e^{-u^2/8}<\infty.
\]
For every $r,t>0$,
\[
  r^m e^{-r^2/(4t)}
  =
  t^{m/2}
  \left(\frac r{\sqrt t}\right)^m
  e^{-r^2/(8t)}e^{-r^2/(8t)}
  \leq
  C_m t^{m/2}e^{-r^2/(8t)}.
\]
Using the distance measure $\nu_\Omega$ defined above, we therefore obtain
\begin{equation}
\label{eq:moment-reduced-to-G}
  G_{\Omega,m}(t)
  \leq
  C_m t^{m/2}G_\Omega(2t).
\end{equation}
It remains to estimate $G_\Omega(2t)$. Formula
\eqref{eq:gaussian-stieltjes}, with $t$ replaced by $2t$, gives
\[
  G_\Omega(2t)
  =
  \frac{1}{4t}
  \int_0^\infty r e^{-r^2/(8t)}
  \vol^{2d}\bigl(E_\Omega(r)\bigr)\,dr.
\]
Splitting the integral at $r_0$, using the assumed bound on
$(0,r_0]$ and $0\leq\vol^{2d}(E_\Omega(r))\leq B_\Omega$ on
$[r_0,\infty)$, and making the change of variables
$u=r^2/(8t)$ in the first integral, we find
\begin{align*}
  G_\Omega(2t)
  &\leq
  \frac{C_0}{4t}
  \int_0^\infty r^{\alpha+1}e^{-r^2/(8t)}\,dr
  +
  \frac{B_\Omega}{4t}
  \int_{r_0}^\infty r e^{-r^2/(8t)}\,dr \\
  &=
  C_0\,8^{\alpha/2}
  \Gamma\left(1+\frac{\alpha}{2}\right)t^{\alpha/2}
  +
  B_\Omega e^{-r_0^2/(8t)}
  =
  O(t^{\alpha/2}).
\end{align*}
Combining this estimate with \eqref{eq:moment-reduced-to-G} proves
\eqref{eq:gaussian-moment-bound}.
}
\end{proof}

\subsection{Small-time asymptotics of the smoothed Kuznecov tail}
\label{subsec:smoothed-tail}

Let $p_t(x,y)$ denote the heat kernel of $(M,g)$. With the sign
convention used above,
\begin{equation}
\label{eq:heat-kernel-spectral-expansion}
  p_t(x,y)
  =
  \sum_{j=0}^{\infty}
  e^{-t\lambda_j^2}e_j(x)e_j(y).
\end{equation}
We shall use the standard local heat-kernel parametrix and the global
Gaussian bound; see, e.g., \cite[Chapter 2]{BGV1992} and
\cite[Chapter 8]{Grigoryan2009}, respectively.

\begin{lemma}[Heat-kernel estimates]
\label{lem:heat-kernel-estimates}
There exist $\rho_0,t_0>0$, with
$2\rho_0<\operatorname{inj}(M)$, and smooth functions $u_k$ on
$\{(x,y)\in M\times M:d_g(x,y)<2\rho_0\}$, $k=0,1,\ldots$, with
$u_0(x,x)=1$, such that for every integer $L\geq1$,
\begin{equation}
\label{eq:heat-parametrix}
  p_t(x,y)
  =
  (4\pi t)^{-{d}/2}e^{-d_g(x,y)^2/(4t)}
  \sum_{k=0}^{L-1}t^ku_k(x,y)
  +
  R_L(t,x,y)
\end{equation}
whenever $0<t\leq t_0$ and $d_g(x,y)<\rho_0$. Moreover, for each $L$
there exists a constant $C_L>0$ such that
\begin{equation}
\label{eq:parametrix-remainder}
  |R_L(t,x,y)|
  \leq C_Lt^{L-{d}/2}
\end{equation}
uniformly for $0<t\leq t_0$ and $d_g(x,y)<\rho_0$. There also exist
constants $c,C>0$ such that
\begin{equation}
\label{eq:global-gaussian-bound}
  0\leq p_t(x,y)
  \leq
  Ct^{-{d}/2}
  \exp\left(-\frac{d_g(x,y)^2}{ct}\right),
  \qquad
  0<t\leq1,\quad x,y\in M.
\end{equation}
Finally, after decreasing $\rho_0$ if necessary,
\begin{equation}
\label{eq:u0-linear-bound}
  |u_0(x,y)-1|
  \leq Cd_g(x,y),
  \qquad d_g(x,y)<\rho_0.
\end{equation}
\end{lemma}

With the heat-kernel estimates established, we now turn to the spectral
side and introduce the smoothed Kuznecov tail $H_\Omega(t)$ by
\begin{equation}
\label{eq:smoothed-tail}
\begin{aligned}
  H_\Omega(t)
  &:=
  \sum_{j=0}^{\infty}
  \bigl(1-e^{-t\lambda_j^2}\bigr)|a_j|^2
  \\
  &=
  \vol(\Omega)
  -
  \sum_{j=0}^{\infty}
  e^{-t\lambda_j^2}|a_j|^2,
  \qquad t>0.
\end{aligned}
\end{equation}

We next relate $H_\Omega(t)$ to the heat kernel $p_t$.

\begin{lemma}
\label{lem:smoothed-tail-representation}
For every $t>0$,
\begin{equation}
\label{eq:smoothed-tail-geometric}
  H_\Omega(t)
  =
  \int_\Omega\int_{M\setminus\Omega}
  p_t(x,y)\,dV_g(y)\,dV_g(x).
\end{equation}
\end{lemma}

\begin{proof}
Since $p_t$ is smooth for $t>0$ and $M$ is compact, the heat semigroup
admits the integral representation
\[
  e^{t\Delta_g}f(x)
  =
  \int_M p_t(x,y)f(y)\,dV_g(y)
\]
for $f\in L^2(M)$. Hence, by the spectral theorem,
\[
  \int_\Omega\int_\Omega
  p_t(x,y)\,dV_g(y)\,dV_g(x)
  =
  \left\langle e^{t\Delta_g}\mathbf 1_\Omega,
  \mathbf 1_\Omega\right\rangle_{L^2(M)}
  =
  \sum_{j=0}^{\infty}
  e^{-t\lambda_j^2}|a_j|^2.
\]
Since $\Delta_g1=0$, we have $e^{t\Delta_g}1=1$, and hence
\[
  \int_M p_t(x,y)\,dV_g(y)=1
  \qquad\text{for every }x\in M.
\]
Therefore
\[
  \int_\Omega\int_M
  p_t(x,y)\,dV_g(y)\,dV_g(x)
  =
  \vol(\Omega).
\]
Subtracting the first identity from the second and using
\eqref{eq:smoothed-tail} gives
\eqref{eq:smoothed-tail-geometric}.
\end{proof}

We next establish the asymptotic relation between $H_\Omega(t)$ and
$G_\Omega(t)$.

\begin{proposition}
\label{prop:gaussian-to-smoothed-tail}
Suppose that
{$\vol^{2d}(E_\Omega(r))=O(r^\alpha)$} as
$r\rightarrow 0^+$.
Then
\begin{equation}
\label{eq:smoothed-tail-gaussian}
  H_\Omega(t)
  =
  (4\pi t)^{-{d}/2}G_\Omega(t)
  +o(t^\beta),
  \qquad t\rightarrow 0^+.
\end{equation}
\end{proposition}

\begin{proof}
Fix $\rho_0,t_0$ as in Lemma \ref{lem:heat-kernel-estimates}, and choose
an integer $L\geq2$ such that $L>\alpha/2$. Set
$\mathcal D_{\rho_0}:=\{(x,y)\in\Omega\times(M\setminus\Omega):
d_g(x,y)<\rho_0\}$.
{Set also
$B_\Omega:=\vol^{2d}(\Omega\times(M\setminus\Omega))
=\vol(\Omega)\vol(M\setminus\Omega)$.}
By \eqref{eq:smoothed-tail-geometric}, we may write
\begin{align}
  H_\Omega(t)
  &=
  \iint_{\mathcal D_{\rho_0}}
  p_t(x,y)\,dV_g(y)\,dV_g(x)
  \notag\\
  &\quad+
  \iint_{\substack{x\in\Omega,\ y\in M\setminus\Omega\\
                   d_g(x,y)\geq\rho_0}}
  p_t(x,y)\,dV_g(y)\,dV_g(x)
  \notag\\
  &=:H_{\mathrm{near}}(t)+H_{\mathrm{far}}(t).
  \label{eq:H-near-far}
\end{align}
The global bound \eqref{eq:global-gaussian-bound} gives
\begin{equation}
\label{eq:H-far}
  0\leq H_{\mathrm{far}}(t)
  \leq
  CB_\Omega t^{-{d}/2}
  \exp\left(-\frac{\rho_0^2}{ct}\right)
  =
  o(t^\beta).
\end{equation}

For $0<t\leq t_0$, insert \eqref{eq:heat-parametrix} into the first
integral in \eqref{eq:H-near-far}. If
\begin{equation}
\label{eq:Ik-definition}
  I_k(t)
  :=
  \iint_{\mathcal D_{\rho_0}}
  e^{-d_g(x,y)^2/(4t)}u_k(x,y)
  \,dV_g(y)\,dV_g(x),
\end{equation}
then
\begin{equation}
\label{eq:H-near-parametrix}
  H_{\mathrm{near}}(t)
  =
  (4\pi t)^{-{d}/2}
  \sum_{k=0}^{L-1}t^kI_k(t)
  +
  \mathcal R_L(t),
\end{equation}
where \eqref{eq:parametrix-remainder} implies
\begin{equation}
\label{eq:integrated-parametrix-remainder}
  |\mathcal R_L(t)|
  \leq
  C_LB_\Omega t^{L-{d}/2}
  =
  o(t^\beta).
\end{equation}
The last equality follows from $L>\alpha/2$ and
$\beta=(\alpha-{d})/2$.

We first consider $I_0(t)$. By \eqref{eq:u0-linear-bound},
\begin{align}
  |I_0(t)-G_\Omega(t)|
  &\leq
  C
  \iint_{\mathcal D_{\rho_0}}
  d_g(x,y)e^{-d_g(x,y)^2/(4t)}
  \,dV_g(y)\,dV_g(x)
  \notag\\
  &\quad+
  \iint_{\substack{x\in\Omega,\ y\in M\setminus\Omega\\
                   d_g(x,y)\geq\rho_0}}
  e^{-d_g(x,y)^2/(4t)}
  \,dV_g(y)\,dV_g(x)
  \notag\\
  &\leq
  CG_{\Omega,1}(t)
  +
  B_\Omega e^{-\rho_0^2/(4t)}.
  \label{eq:I0-minus-G}
\end{align}
Lemma \ref{lem:gaussian-moment-bound} therefore gives
\begin{equation}
\label{eq:I0-contribution}
  (4\pi t)^{-{d}/2}I_0(t)
  =
  (4\pi t)^{-{d}/2}G_\Omega(t)
  +
  O(t^{\beta+1/2})
  +
  O\bigl(t^{-{d}/2}e^{-\rho_0^2/(4t)}\bigr).
\end{equation}

For $1\leq k\leq L-1$, the function $u_k$ is bounded on the closure of
the fixed neighborhood $\{d_g(x,y)<\rho_0\}$. Hence
\begin{align}
  \left|
  (4\pi t)^{-{d}/2}t^kI_k(t)
  \right|
  &\leq
  C_kt^{k-{d}/2}G_\Omega(t)
  \notag\\
  &=
  O(t^{\beta+k})
  =
  o(t^\beta).
  \label{eq:higher-parametrix-terms}
\end{align}
Collecting the preceding estimates, we obtain
\begin{align*}
  H_\Omega(t)
  &=
  (4\pi t)^{-{d}/2}G_\Omega(t)
  +
  O(t^{\beta+1/2})
  +
  \sum_{k=1}^{L-1}O(t^{\beta+k})
  \\
  &\quad+
  O(t^{L-{d}/2})
  +
  O\bigl(t^{-{d}/2}e^{-\rho_0^2/(4t)}\bigr)
  +
  O\bigl(t^{-{d}/2}e^{-\rho_0^2/(ct)}\bigr).
\end{align*}
Every error term in the last two lines is $o(t^\beta)$. This proves
\eqref{eq:smoothed-tail-gaussian}.
\end{proof}

\subsection{The Kuznecov tail}
\label{subsec:tauberian}

In this subsection, we use a classical Tauberian theorem to pass from the
smoothed Kuznecov tail to the sharp Kuznecov tail.

For the squared spectral parameter $S\geq0$, define
\begin{equation}
\label{eq:sharp-tail}
  \mathcal T_\Omega(S)
  :=
  \sum_{\lambda_j^2>S}|a_j|^2
  =
  \vol(\Omega)-N_\Omega(\sqrt S).
\end{equation}
The function $\mathcal T_\Omega$ is bounded, nonnegative,
nonincreasing, and right-continuous. Moreover,
$\mathcal T_\Omega(S)\to0$ as $S\to\infty$ by
\eqref{eq:counting-and-parseval}.

The following elementary identity explains why $H_\Omega(t)$ is
naturally viewed as a smoothed Kuznecov tail.

\begin{lemma}
\label{lem:tail-laplace}
For every $t>0$,
\begin{equation}
\label{eq:tail-laplace}
  H_\Omega(t)
  =
  t\int_0^\infty
  e^{-tS}\mathcal T_\Omega(S)\,dS.
\end{equation}
\end{lemma}

\begin{proof}
All summands are nonnegative. Tonelli's theorem therefore gives
\begin{align*}
  t\int_0^\infty
  e^{-tS}\mathcal T_\Omega(S)\,dS
  &=
  t\sum_{j=0}^{\infty}|a_j|^2
  \int_0^{\lambda_j^2}e^{-tS}\,dS
  \\
  &=
  \sum_{j=0}^{\infty}
  \bigl(1-e^{-t\lambda_j^2}\bigr)|a_j|^2
  =
  H_\Omega(t).
\end{align*}
\end{proof}

{For the precise asymptotic, we use the following tail form of
Karamata's theorem.}

\begin{theorem}[{Karamata's theorem for\allowbreak\ monotone tails}]
\label{thm:monotone-tauberian}
{
Let $T:[0,\infty)\to[0,\infty)$ be nonincreasing, let
$0<\beta<1$, and let $A>0$. Then
\[
  t\int_0^\infty e^{-tS}T(S)\,dS
  =
  At^\beta+o(t^\beta),
  \qquad t\rightarrow 0^+,
\]
if and only if
\[
  T(S)
  =
  \frac{A}{\Gamma(1-\beta)}S^{-\beta}
  +o(S^{-\beta}),
  \qquad S\to\infty.
\]
}
\end{theorem}
The implication from the second asymptotic to the first is Abelian.
Indeed, the change of variables $u=tS$ gives
\begin{align*}
  t\int_0^\infty e^{-tS}T(S)\,dS
  &=
  \int_0^\infty e^{-u}T(u/t)\,du
  \\
  &=
  \frac{A}{\Gamma(1-\beta)}t^\beta
  \int_0^\infty e^{-u}u^{-\beta}\,du
  +o(t^\beta)
  \\
  &=
  At^\beta+o(t^\beta).
\end{align*}
The error estimate follows by splitting the integral at $u=tS_0$ and
applying dominated convergence; the condition $0<\beta<1$ ensures that
$u^{-\beta}$ is integrable at the origin.

The reverse implication is Tauberian. Applying Karamata's Tauberian
theorem to
\[
  U(S):=\int_0^S T(u)\,du
\]
and then the monotone density theorem yields the stated asymptotic for
$T$. See \cite[Theorems 1.7.1 and 1.7.2]{BGT1987}.

The two-sided estimate needed below follows instead from a direct
comparison argument.

\begin{corollary}[Comparison form]
\label{cor:tauberian-comparison}
Let $T:[0,\infty)\to[0,\infty)$ be nonincreasing, and let
$0<\beta<1$. If there exist constants $c,C>0$ such that
\[
  ct^\beta
  \leq
  t\int_0^\infty e^{-tS}T(S)\,dS
  \leq
  Ct^\beta
\]
for all sufficiently small $t>0$, then there exist constants
$c',C'>0$ such that
\[
  c'S^{-\beta}
  \leq T(S)
  \leq C'S^{-\beta}
\]
for all sufficiently large $S$.
\end{corollary}

\begin{proof}
Set $L(t):=t\int_0^\infty e^{-tS}T(S)\,dS$. For all sufficiently
large $R$, monotonicity and the assumed upper bound give
\[
e^{-1}T(R)
\leq
\frac{1}{R}\int_0^R e^{-S/R}T(S)\,dS
\leq
L(R^{-1})
\leq
CR^{-\beta}.
\]
Thus $T(R)\leq eCR^{-\beta}$.

For the reverse inequality, choose $\theta\in(0,1)$ so that
$eC\theta^{1-\beta}\leq c/2$. For all sufficiently large $R$,
\[
\frac{1}{R}\int_0^{\theta R}T(S)\,dS
\leq
e\theta L\bigl((\theta R)^{-1}\bigr)
\leq
eC\theta^{1-\beta}R^{-\beta}
\leq
\frac{c}{2}R^{-\beta}.
\]
Monotonicity also gives
\[
L(R^{-1})
\leq
\frac{1}{R}\int_0^{\theta R}T(S)\,dS
+
e^{-\theta}T(\theta R).
\]
Combining this with $L(R^{-1})\geq cR^{-\beta}$, we obtain
$T(\theta R)\geq \frac{c}{2}e^\theta R^{-\beta}$. Replacing
$\theta R$ by $S$ proves the lower bound.
\end{proof}

\begin{proof}[Proof of Theorem \ref{thm:crossing-volume-kuznecov}]
{
Suppose first that
$0<\mathcal V^s(\partial\Omega)<\infty$. By
\eqref{eq:crossing-volume-definition},
\[
  \vol^{2d}\bigl(E_\Omega(r)\bigr)
  =
  \frac{\vol(B^{d-1})}{d+1}\,
  \mathcal V^s(\partial\Omega)r^{2{d}-s}
  +o\bigl(r^{2{d}-s}\bigr),
  \qquad r\rightarrow0^+.
\]
With $\alpha$ and $\beta$ as defined above, Lemma \ref{lem:gaussian-average}
and the change of variables $r=\sqrt t\,u$ give
\[
  G_\Omega(t)
  =
  \frac{\vol(B^{d-1})}{d+1}\,
  \mathcal V^s(\partial\Omega)
  2^\alpha\Gamma\left(1+\frac{\alpha}{2}\right)
  t^{\alpha/2}
  +o\bigl(t^{\alpha/2}\bigr).
\]
Here the error term follows by splitting the integral in
\eqref{eq:gaussian-stieltjes} at a fixed sufficiently small radius and
using the Gaussian decay on the remaining part. Proposition \ref{prop:gaussian-to-smoothed-tail}
therefore yields
\[
  H_\Omega(t)=A_\Omega t^\beta+o(t^\beta),
\]
where
\[
  A_\Omega
  =\frac{\vol(B^{d-1})}{d+1}\,
  \mathcal V^s(\partial\Omega)\,
  2^{{d}-s}\pi^{-{d}/2}
  \Gamma\left(1+{d}-\frac{s}{2}\right).
\]
Lemma \ref{lem:tail-laplace} and
Theorem \ref{thm:monotone-tauberian} now give
\[
  \mathcal T_\Omega(S)
  =
  \frac{A_\Omega}{\Gamma(1-\beta)}S^{-\beta}
  +o(S^{-\beta}).
\]
Substituting $S=\lambda^2$ in \eqref{eq:sharp-tail} proves
\eqref{eq:intro-two-term-kuznecov}, with $C_\Omega$ given by
\eqref{eq:intro-kuznecov-constant}.

Conversely, suppose that \eqref{eq:intro-two-term-kuznecov} holds for
some $C_\Omega>0$. Then
\[
  \mathcal T_\Omega(S)
  =
  C_\Omega S^{-\beta}+o(S^{-\beta}).
\]
The Abelian implication in
Theorem \ref{thm:monotone-tauberian}, together with
Lemma \ref{lem:tail-laplace}, yields
\begin{equation}
\label{eq:reverse-smoothed-tail}
  H_\Omega(t)
  =
  C_\Omega\Gamma(1-\beta)t^\beta
  +o(t^\beta).
\end{equation}
The local parametrix in
Lemma \ref{lem:heat-kernel-estimates} implies that, for all sufficiently
small $r>0$,
\[
  p_{r^2}(x,y)\geq cr^{-{d}}
  \qquad\text{whenever }d_g(x,y)<r.
\]
Hence Lemma \ref{lem:smoothed-tail-representation} gives
\[
  H_\Omega(r^2)
  \geq
  cr^{-{d}}\vol^{2d}\bigl(E_\Omega(r)\bigr).
\]
It follows from \eqref{eq:reverse-smoothed-tail} that
$\vol^{2d}(E_\Omega(r))
=O(r^{{d}+2\beta})=O(r^\alpha)$. We may therefore apply
Proposition \ref{prop:gaussian-to-smoothed-tail}; together with
\eqref{eq:reverse-smoothed-tail}, it gives
\[
  G_\Omega(t)
  =
  (4\pi)^{{d}/2}C_\Omega\Gamma(1-\beta)t^{\alpha/2}
  +o(t^{\alpha/2}).
\]
Set
\[
  U(x):=\vol^{2d}\bigl(E_\Omega(2\sqrt{x})\bigr),
  \qquad x\geq0.
\]
By \eqref{eq:E-as-distribution}, $U$ is the distribution function of
the pushforward of $\nu_\Omega$ under the map $r\mapsto r^2/4$.
Consequently,
\[
  G_\Omega(t)
  =
  \int_{[0,\infty)}e^{-x/t}\,dU(x).
\]
The local form of Karamata's theorem for Laplace--Stieltjes transforms
now implies
\[
  U(x)
  =
  \frac{(4\pi)^{{d}/2}C_\Omega\Gamma(1-\beta)}
       {\Gamma(1+\alpha/2)}
  x^{\alpha/2}
  +o\bigl(x^{\alpha/2}\bigr),
  \qquad x\rightarrow0^+;
\]
see \cite[Theorem 1.7.1]{BGT1987}. Taking $x=r^2/4$, we obtain
\[
  \vol^{2d}\bigl(E_\Omega(r)\bigr)
  =
  2^{s-{d}}\pi^{{d}/2}
  \frac{
    \Gamma\left(1-\frac{{d}-s}{2}\right)
  }{
    \Gamma\left(1+{d}-\frac{s}{2}\right)
  }
  C_\Omega r^\alpha
  +o(r^\alpha),
  \qquad r\rightarrow0^+.
\]
It follows from \eqref{eq:crossing-volume-definition} that
\[
  \mathcal V^s(\partial\Omega)
  =
  \frac{d+1}{\vol(B^{d-1})}\,
  2^{s-{d}}\pi^{{d}/2}
  \frac{
    \Gamma\left(1-\frac{{d}-s}{2}\right)
  }{
    \Gamma\left(1+{d}-\frac{s}{2}\right)
  }
  C_\Omega.
\]
Thus the crossing volume exists, is positive and finite, and the last
identity is equivalent to \eqref{eq:intro-kuznecov-constant}.
}
\end{proof}

\begin{proof}[Completion of the proof of Theorem \ref{thm:kuznecov-two-sided}]
{
We first combine Proposition \ref{prop:crossing-volume} with
Lemma \ref{lem:gaussian-average} to obtain two-sided bounds of order
$t^{\alpha/2}$ for $G_\Omega(t)$. We then use
Proposition \ref{prop:gaussian-to-smoothed-tail} to transfer these estimates
to $H_\Omega(t)$, obtaining two-sided bounds of order $t^\beta$. Next, we
apply Lemma \ref{lem:tail-laplace} and
Corollary \ref{cor:tauberian-comparison} to deduce two-sided bounds of order
$S^{-\beta}$ for $\mathcal T_\Omega(S)$. Finally, setting $S=\lambda^2$ in
\eqref{eq:sharp-tail} completes the proof.
}
\end{proof}

Finally, we prove Corollary \ref{cor:codimension-one-kuznecov} by establishing the
following geometric identity.

\begin{proposition}[Crossing volume in codimension one]
\label{prop:codimension-one-crossing-volume}
Let $\Omega\subset M$ be a domain of finite perimeter. Then
\[
\mathcal V^{d-1}(\partial\Omega)
=
\vol(\partial^*\Omega),
\]
where $\partial^*\Omega$ denotes the reduced boundary.
In particular, if $\Omega$ is a Lipschitz domain, then
\[
\mathcal V^{d-1}(\partial\Omega)=\vol(\partial\Omega).
\]
\end{proposition}

\begin{proof}
Let $\chi_\Omega$ be the indicator function of $\Omega$. By symmetry,
\[
2\vol^{2d}\bigl(E_\Omega(r)\bigr)
=
\int_M\int_M
\mathbf 1_{\{d_g(x,y)<r\}}
\bigl|\chi_\Omega(x)-\chi_\Omega(y)\bigr|
\,dV_g(y)\,dV_g(x).
\]
By Dávila's Euclidean $BV$ difference-quotient formula
\cite{Davila2002}, localized in normal coordinate charts, we obtain
\[
\lim_{r\to0^+}r^{-(d+1)}
\int_M\int_M
\mathbf 1_{\{d_g(x,y)<r\}}
|u(x)-u(y)|\,dV_g(y)\,dV_g(x)
=
\frac{2\vol(B^{d-1})}{d+1}|Du|(M).
\]
Since
\[
\int_{B^d}|z_1|\,dz
=
\frac{2\vol(B^{d-1})}{d+1}
\]
and
\[
|D\chi_\Omega|(M)=\vol(\partial^*\Omega),
\]
we obtain
\[
\lim_{r\to0^+}r^{-(d+1)}
\vol^{2d}\bigl(E_\Omega(r)\bigr)
=
\frac{\vol(B^{d-1})}{d+1}\,
\vol(\partial^*\Omega).
\]
The definition of $\mathcal V^{d-1}$ therefore gives
\[
\mathcal V^{d-1}(\partial\Omega)
=
\vol(\partial^*\Omega).
\]
For a Lipschitz domain, $\partial^*\Omega$ agrees with $\partial\Omega$
up to a set of $(d-1)$-dimensional volume zero, which proves the last
statement.
\end{proof}

\section{Kuznecov asymptotics for self-similar boundaries}
\label{sec:self-similar-crossing-volumes}

In this section, we introduce the precise self-similar framework and
establish the corresponding crossing-volume asymptotics. The non-lattice
case proves Theorem \ref{thm:admissible-self-similar-kuznecov}, while the
lattice case will be used in Section \ref{sec:lattice-case}.

\subsection{Admissible self-similar boundaries}
\label{subsec:admissible-self-similar-boundaries}

An \emph{iterated function system} (IFS) on $\mathbb R^d$ is a finite family
$\{S_1,\ldots,S_N\}$ of contracting similarities,
$S_i(x)=r_iT_ix+p_i$, with $0<r_i<1$ and $T_i\in O(d)$. Its
self-similar attractor is the unique nonempty compact set
$F=\bigcup_{i=1}^N S_iF$.

The IFS is said to satisfy the \emph{open set condition} (OSC) if there
exists a bounded nonempty open set $O$ such that $S_iO\subset O$ for all
$i$ and $S_iO\cap S_jO=\varnothing$ whenever $i\ne j$.
For an admissible local self-similar boundary, we fix such a feasible
open set $O$ which is also  assumed to be connected. If $s$ is determined by $\sum_{i=1}^N r_i^s=1$, then
$s=\dim_{\mathrm H}F$ under the OSC. See
\cite{Hutchinson1981,Falconer2014}.

To use $F$ as a local model for a domain boundary, we impose the following
conditions.

\begin{definition}[Admissible local self-similar boundary]
\label{def:admissible-local-self-similar-boundary}
Under the preceding notation, let $s<d$. We call $F$ an
\emph{admissible local self-similar boundary} of dimension $s$ if:
\begin{enumerate}
\item[(i)] $F\subset\overline O$ and
\[
\overline{\dim}_M(\partial O)<s;
\]

\item[(ii)] {there exist nonempty disjoint open sets
$\Omega^+$ and $\Omega^-$ such that
\[
O\setminus F=\Omega^+\sqcup\Omega^-,
\]
and
\[
F\cap O
=
O\cap\partial\Omega^+
=
O\cap\partial\Omega^-.
\]}

\item[(iii)] for every $i$, the map $S_i$ either preserves or interchanges
the two {sides}:
\[
S_i\Omega^\pm\subset\Omega^\pm
\qquad\text{or}\qquad
S_i\Omega^\pm\subset\Omega^\mp.
\]
\end{enumerate}
The IFS, the feasible open set $O$, and the decomposition
$O\setminus F=\Omega^+\sqcup\Omega^-$ are regarded as part of the
defining data.
\end{definition}

Since the boundary of a nonempty bounded open set in $\mathbb R^d$ has
upper Minkowski dimension at least $d-1$, condition \textup{(i)} implies
$s>d-1$. Thus throughout this self-similar setting,
\[
d-1<s<d.
\]
We now assemble the local models into the boundary of a domain.

\begin{definition}[Admissible self-similar boundary]
\label{def:admissible-self-similar-boundary}
Let $\Omega\subset\mathbb R^d$ be a bounded domain and
$\Gamma:=\partial\Omega$. We call $\Gamma$ an
\emph{admissible self-similar boundary} of dimension $s$ if there exist
finitely many admissible local self-similar boundaries
$F_1,\ldots,F_m$ of dimension $s$, with feasible open sets
$O_1,\ldots,O_m$, such that:
\begin{enumerate}
\item[(i)]
\[
\Gamma=\bigcup_{i=1}^m F_i;
\]

\item[(ii)] the sets $O_i$ are pairwise disjoint and
$F_i\subset\overline{O_i}$ for every $i$;

\item[(iii)] after interchanging the two local {sides} if necessary,
\[
\Omega_i^+=\Omega\cap O_i,
\qquad
\Omega_i^-=(\mathbb R^d\setminus\overline\Omega)\cap O_i.
\]
\end{enumerate}
One such decomposition is fixed throughout the paper.
\end{definition}

{
\begin{remark}[The $(3,\tfrac14)$-von Koch snowflake]
\label{rem:nc-snowflake-admissibility}
We verify the admissibility conditions for the $(3,\tfrac14)$-von Koch
snowflake. Similar verifications apply to a large class of self-avoiding
$(n,c)$-von Koch snowflakes, including lattice examples.

Let $F$ be the generating arc of the $(3,\tfrac14)$-von Koch snowflake,
based on $[0,1]$. By \cite[Theorem 3.1]{KeletiPaquette2010}, $F$ is a
simple arc, and its defining IFS satisfies the open set condition. The
polygon constructed in the proof yields a polygonal feasible open set $O$.
Hence
$\overline{\dim}_M(\partial O)=1<s$, which verifies condition
\textup{(i)} of Definition \ref{def:admissible-local-self-similar-boundary}.

Three rotated copies $F_j$ of $F$ form the Jordan curve $\Gamma$ bounding
the snowflake domain $\Omega$.
Orient $\Gamma$ positively and give each $F_j$ the induced orientation.
For the corresponding feasible open sets $O_j$, we prescribe the local
sides by
\[
\Omega_j^+:=\Omega\cap O_j,
\qquad
\Omega_j^-:=(\mathbb R^2\setminus\overline\Omega)\cap O_j.
\]
These are the interior and exterior sides inherited from $\Gamma$.
Since $\Gamma\cap O_j=F_j\cap O_j$, they partition $O_j\setminus F_j$
and have common relative boundary $F_j\cap O_j$, as required in
condition \textup{(ii)} of
Definition \ref{def:admissible-local-self-similar-boundary}.
Each defining similarity maps the oriented generating arc to its
corresponding subarc. See \cite[p. 125]{KeletiPaquette2010}. Its
orthogonal part preserves or interchanges the left and right sides.
Since the first-level cell contains no other part of $\Gamma$, the
two image sides agree, up to interchange, with the prescribed sides in
that cell. This verifies condition \textup{(iii)} of the same definition.

The three sets $O_j$ are pairwise disjoint and
$\Gamma=\bigcup_{j=1}^3F_j$. The choice of $\Omega_j^\pm$ above also
gives condition \textup{(iii)} of
Definition \ref{def:admissible-self-similar-boundary}. Thus $\Gamma$ is
an admissible self-similar boundary. Finally, its contraction ratios
are $1/4,1/4,3/8,3/8$, and
$\log(1/4)/\log(3/8)\notin\mathbb Q$. Thus $\Gamma$ is non-lattice
and satisfies the hypotheses of
Theorem \ref{thm:admissible-self-similar-kuznecov}.
\end{remark}
}

For the spectral results, we regard these Euclidean domains as domains in
the flat torus $\mathbb T^d$. After a common translation and dilation, the entire configuration may be
placed inside a Euclidean ball of diameter smaller than the injectivity
radius of $\mathbb T^d$, so that Euclidean and flat-torus distances agree
on the configuration. 
The transformation preserves the contraction ratios and dimensions and
affects only the multiplicative constants below. We henceforth identify
the Euclidean configuration with its image in $\mathbb T^d$ and write
$|x-y|$ for both the Euclidean and flat-torus distances on
$\bigcup_iO_i$.

The conditions above also imply that distinct local pieces meet only in
lower-dimensional sets. Indeed, for $i\ne j$,
$F_i\cap F_j\subset\partial O_i\cap\partial O_j$, and hence
\[
\overline{\dim}_M(F_i\cap F_j)<s.
\]

An IFS with contraction ratios $r_1,\ldots,r_N$ is called \emph{lattice}
if there exists $h>0$ such that $-\log r_i\in h\mathbb Z$ for every $i$,
and \emph{non-lattice} otherwise. In the lattice case, the largest such
$h$ is called the \emph{maximal lattice span}. An admissible local
self-similar boundary is called lattice or non-lattice according to its
generating IFS. An admissible self-similar boundary is called lattice,
respectively non-lattice, if all local pieces in its fixed decomposition
are lattice, respectively non-lattice. These notions are unchanged under
the translation and dilation used in the preceding torus embedding.

The main ingredient in the proof of
Theorem \ref{thm:admissible-self-similar-kuznecov} is the following
renewal theorem from \cite[Section 2.3, Renewal Theorem, pp. 197--199]{LevitinVassiliev1996}. 

\begin{theorem}[Renewal theorem]
\label{thm:adapted-renewal}
Let $0<y_1\leq\cdots\leq y_N$, let $c_i>0$ satisfy
$\sum_{i=1}^N c_i=1$, and set
\[
  \eta:=\sum_{i=1}^N c_i y_i.
\]
Suppose that $z:\mathbb R\to\mathbb R$ is piecewise continuous, that its
points of discontinuity have no finite accumulation point, and that
\[
  |z(t)|\leq Ce^{-\tau|t|},
  \qquad t\in\mathbb R,
\]
for some $C,\tau>0$. Then the renewal equation
\[
  Z(t)=\sum_{i=1}^N c_iZ(t-y_i)+z(t),
  \qquad t\in\mathbb R,
\]
has a unique solution satisfying $Z(t)\to0$ as $t\to-\infty$, and this
solution is bounded on $\mathbb R$.

If $\{y_1,\ldots,y_N\}$ is non-arithmetic, then
\[
  Z(t)
  =
  \frac1\eta\int_{\mathbb R}z(u)\,du+o(1),
  \qquad t\to\infty.
\]
If $\{y_1,\ldots,y_N\}$ is arithmetic with maximal span $h>0$, then
\begin{equation}
\label{eq:arithmetic-renewal-profile}
  Z(t)=p(t)+o(1),
  \qquad t\to\infty,
\end{equation}
where
\begin{equation}
\label{eq:arithmetic-renewal-periodization}
  p(t):=
  \frac{h}{\eta}
  \sum_{\ell\in\mathbb Z}z(t-\ell h)
\end{equation}
is a bounded $h$-periodic function. If $z$ is left-continuous,
right-continuous, or continuous, then $Z$ and, in the arithmetic case,
$p$ have the same corresponding continuity property.
\end{theorem}

In the application below, after relabeling if necessary, we take
$y_i=-\log r_i$ and $c_i=r_i^s$.
The non-lattice conclusion will be used to prove
Theorem \ref{thm:admissible-self-similar-kuznecov}, while the lattice
conclusion will be used in Section \ref{sec:lattice-case}. We first
derive the required renewal equation from the local two-body set of an
admissible local self-similar boundary and estimate its forcing term
$\varphi_{2d}$.

\subsection{A local renewal equation}
\label{subsec:local-renewal}
We retain the notation introduced above. Let $F$ be an
\emph{admissible local self-similar boundary} of dimension $s$, generated
by an IFS $\{S_1,\ldots,S_N\}$ with contraction ratios
$r_1,\ldots,r_N$, and let $O$ be its fixed feasible open set. Write
$O\setminus F=\Omega^+\sqcup\Omega^-$, and define
\[
E_O^{\Omega^+}(r)
:=
\bigl\{(x,y)\in\Omega^+\times\Omega^-:|x-y|<r\bigr\}.
\]
Set $f(r):=\vol^{2d}\bigl(E_O^{\Omega^+}(r)\bigr)$. For an open set
$D\subset\mathbb T^d$, we also define
\[
E_{\mathbb T^d}^{D}(r)
:=
\bigl\{(x,y)\in
D\times(\mathbb T^d\setminus D):|x-y|<r\bigr\}.
\]

\begin{proposition}[Local renewal equation]
\label{thm:local-renewal}
Let $F$ be an {admissible local self-similar boundary} of dimension $s$, generated
by an IFS with contraction ratios $r_1,\ldots,r_N$, and let
$O$, $\Omega^+$, and $\Omega^-$ be as above. Set
$f(r)=\vol^{2d}\bigl(E_O^{\Omega^+}(r)\bigr)
=\vol^{2d}\bigl(E_O^{\Omega^-}(r)\bigr)$. Then there exist
$\gamma>0$, $C>0$, and a
nonnegative bounded function
$\varphi_{2d}:(0,\infty)\to[0,\infty)$, with a discrete set of
discontinuities, such that, for every $r>0$,
\begin{equation}
\label{eq:local-renewal}
f(r)
=
\sum_{i=1}^N r_i^{2d}
\mathbf 1_{(0,r_i]}(r)
f\left(\frac r{r_i}\right)
+\varphi_{2d}(r),
\end{equation}
and
\[
0\leq\varphi_{2d}(r)\leq Cr^{2d-s+\gamma},
\qquad r>0.
\]
\end{proposition}

\begin{proof}
Define $\varphi_{2d}(r):=
f(r)-\sum_{i=1}^N r_i^{2d}\mathbf 1_{(0,r_i]}(r)
f(r/r_i)$.
Then \eqref{eq:local-renewal} holds by definition. We first verify the
condition on the discontinuities of $\varphi_{2d}$. For every $r>0$,
the level set $\{(x,y):|x-y|=r\}$ has zero $2d$-dimensional Lebesgue
measure. Dominated convergence applied to the integral defining $f$
therefore shows that $f$ is continuous on $(0,\infty)$. It follows
that the only possible discontinuities of $\varphi_{2d}$ occur at the
finitely many points $r_1,\ldots,r_N$. In particular, its set of
discontinuities is discrete.

Since $O$ is bounded, $f$ is bounded, and hence so is
$\varphi_{2d}$. Once $\gamma>0$ has been chosen below, we have
$2d-s+\gamma>0$. The desired estimate for
$|\varphi_{2d}(r)|$ is therefore automatic for $r\geq r_{\min}$, where
$r_{\min}:=\min_{1\leq i\leq N}r_i$. It remains to establish the
estimate for $0<r<r_{\min}$; in this range all the cutoff functions in
the definition of $\varphi_{2d}$ are equal to one.

We now derive the untruncated renewal identity needed for this estimate.

To identify the two sides inside each first-level cell, observe first that
the open set condition implies
$S_iO\cap\overline{S_jO}=\varnothing$ whenever $i\ne j$. Since
$S_jF\subset\overline{S_jO}$, it follows that
$F\cap S_iO=S_i(F\cap O)$. {Therefore
\[
S_iO\setminus F=S_i\Omega^+\sqcup S_i\Omega^-.
\]
The side-compatibility condition places these two sets in opposite
prescribed sides of $O\setminus F$.} Thus, for each
$i$, we may choose
$\varepsilon_i\in\{+,-\}$ so that
$S_i\Omega^{\varepsilon_i}=S_iO\cap\Omega^+$ and
$S_i\Omega^{-\varepsilon_i}=S_iO\cap\Omega^-$, where
$-\varepsilon_i$ denotes the sign opposite to $\varepsilon_i$.

For the remainder of the proof, set
\[
\Omega_i^+:=S_i\Omega^{\varepsilon_i}=S_iO\cap\Omega^+,
\qquad
\Omega_i^-:=S_i\Omega^{-\varepsilon_i}=S_iO\cap\Omega^-.
\]
All subsequent set decompositions use these ordered local sides, and we shall
not distinguish the cases $\varepsilon_i=+$ and $\varepsilon_i=-$.
For each $i$, define
\[
E_{S_iO}^{\Omega_i^+}(r)
:=
\bigl\{(x,y)\in\Omega_i^+\times\Omega_i^-:|x-y|<r\bigr\}.
\]
We decompose $E_O^{\Omega^+}(r)$ into the pairs whose two coordinates lie in the
same first-level cell and the remaining pairs:
\[
E_O^{\Omega^+}(r)
=
\bigsqcup_{i=1}^N E_{S_iO}^{\Omega_i^+}(r)
\sqcup B(r),
\]
where $B(r)$ consists of the pairs in $E_O^{\Omega^+}(r)$ that belong to none of
the sets $E_{S_iO}^{\Omega_i^+}(r)$. Taking product volumes gives
\[
f(r)
=
\sum_{i=1}^N\vol^{2d}\bigl(E_{S_iO}^{\Omega_i^+}(r)\bigr)
+\vol^{2d}\bigl(B(r)\bigr).
\]
The map $S_i\times S_i$ scales $\vol^{2d}$ by
$r_i^{2d}$ and distances by $r_i$. According to the ordering of the
local sides, it maps one of the two sets
$E_O^{\Omega^+}(r/r_i)$ and $E_O^{\Omega^-}(r/r_i)$ onto
$E_{S_iO}^{\Omega_i^+}(r)$. Since the two source sets have the same
product volume,
\[
\vol^{2d}\bigl(E_{S_iO}^{\Omega_i^+}(r)\bigr)
=r_i^{2d}f(r/r_i),
\]
and consequently
\[
f(r)
=
\sum_{i=1}^N r_i^{2d}f\left(\frac r{r_i}\right)
+\vol^{2d}\bigl(B(r)\bigr).
\]
Comparing this identity with the definition of $\varphi_{2d}$, we obtain
\[
\varphi_{2d}(r)
=
\vol^{2d}\bigl(B(r)\bigr)
+
\sum_{i=1}^N r_i^{2d}
\mathbf 1_{(r_i,\infty)}(r)
f\left(\frac r{r_i}\right).
\]
Thus $\varphi_{2d}(r)\geq0$ for every $r>0$, while
$\varphi_{2d}(r)=\vol^{2d}\bigl(B(r)\bigr)$ whenever
$0<r<r_{\min}$.
It remains to prove that, for some $\gamma>0$,
\[
\vol^{2d}\bigl(B(r)\bigr)=O\bigl(r^{2d-s+\gamma}\bigr),
\qquad r\rightarrow 0^+.
\]

We now refine the decomposition of $B(r)$. Set
$H=\bigcup_{i=1}^N\overline{S_iO}$ and $G=O\setminus H$, and define
\[
  B_0(r):=\{(x,y)\in B(r):x,y\in G\},
\]
$B_1(r):=\{(x,y)\in B(r):x\in H\}$ and
$B_2(r):=\{(x,y)\in B(r):y\in H\}$.
Then
\[
  B(r)=B_0(r)\cup B_1(r)\cup B_2(r),
\]
where the union need not be disjoint. We first estimate $B_0(r)$.
Since
\[
  F=\bigcup_{i=1}^N S_iF
  \subset
  \bigcup_{i=1}^N\overline{S_iO}
  =H,
\]
we have $F\cap G=\varnothing$, and every connected component of $G$
is contained in either $\Omega^+$ or $\Omega^-$. If
$(x,y)\in B_0(r)$, then $x\in\Omega^+$ and
$y\in\Omega^-$; hence $x$ and $y$ lie in different connected
components of $G$. If $\operatorname{dist}(x,\partial G)\geq r$,
then the ball $B_{\mathbb T^d}(x,r)$ is contained in $G$. Since
$|x-y|<r$, this ball contains both $x$ and $y$, contradicting the
preceding observation. Consequently,
\[
  B_0(r)
  \subset
  \bigl\{(x,y):x\in(\partial G)_r\cap G,\ |x-y|<r\bigr\},
\]
and therefore
\begin{equation}
\label{eq:B-zero-boundary-bound}
  \vol^{2d}\bigl(B_0(r)\bigr)
  \leq C r^d \vol\bigl((\partial G)_r\bigr).
\end{equation}

For each $i$, let
$B_{1,i}(r):=\{(x,y)\in B_1(r):x\in\overline{S_iO}\}$ and
$B_{2,i}(r):=\{(x,y)\in B_2(r):y\in\overline{S_iO}\}$.
Then $B_1(r)=\bigcup_{i=1}^N B_{1,i}(r)$ and
$B_2(r)=\bigcup_{i=1}^N B_{2,i}(r)$.

We first estimate $B_{1,i}(r)$. Define
$\Sigma_i(r)=\{(x,y)\in B_{1,i}(r):x\in\partial(S_iO)\}$.
Since $\partial(S_iO)=S_i(\partial O)$ and
$\overline{\dim}_M(\partial O)<s<d$, the set $\partial(S_iO)$ has zero
Lebesgue measure. Therefore $\vol^{2d}\bigl(\Sigma_i(r)\bigr)=0$. If
$(x,y)\in B_{1,i}(r)\setminus\Sigma_i(r)$, then $x\in S_iO$. Moreover,
$y\notin S_iO$; otherwise both coordinates would lie in the same
first-level cell, and the pair would belong to
$E_{S_iO}^{\Omega_i^+}(r)$,
contrary to the definition of $B(r)$. Consequently,
$B_{1,i}(r)\subset
E_{\mathbb T^d}^{S_iO}(r)\cup\Sigma_i(r)$, where we have used
$O\setminus S_iO\subset\mathbb T^d\setminus S_iO$. Hence
$\vol^{2d}\bigl(B_{1,i}(r)\bigr)
\leq\vol^{2d}\bigl(E_{\mathbb T^d}^{S_iO}(r)\bigr)$, and summing over $i$
gives
$\vol^{2d}\bigl(B_1(r)\bigr)
\leq\sum_{i=1}^N\vol^{2d}\bigl(E_{\mathbb T^d}^{S_iO}(r)\bigr)$.
Interchanging the two coordinates and applying the same argument to
$B_{2,i}(r)$, we obtain
$\vol^{2d}\bigl(B_2(r)\bigr)
\leq\sum_{i=1}^N\vol^{2d}\bigl(E_{\mathbb T^d}^{S_iO}(r)\bigr)$.

Choose $q$ such that
$\overline{\dim}_M(\partial O)<q<s$, and set $\gamma=s-q>0$. We then
have
\[
  \partial G
  \subset
  \partial O\cup\partial H
  \subset
  \partial O\cup\bigcup_{i=1}^N\partial(S_iO).
\]
Since $\partial(S_iO)=S_i(\partial O)$, it follows that
$\overline{\dim}_M(\partial G)<q$, and hence
$\vol\bigl((\partial G)_r\bigr)=O(r^{d-q})$ as $r\rightarrow0^+$. Thus
\eqref{eq:B-zero-boundary-bound} gives
\[
  \vol^{2d}\bigl(B_0(r)\bigr)
  =O(r^{2d-q})
  =O(r^{2d-s+\gamma}).
\]
Moreover,
$\overline{\dim}_M(\partial(S_iO))
=\overline{\dim}_M(\partial O)<q$, and
Remark \ref{rem:minkowski-upper-bound} gives
$\vol^{2d}\bigl(E_{\mathbb T^d}^{S_iO}(r)\bigr)
=O(r^{2d-q})=O(r^{2d-s+\gamma})$ as $r\rightarrow 0^+$.
Since
\begin{multline}
    \vol^{2d}\bigl(B(r)\bigr)
  \leq
  \vol^{2d}\bigl(B_0(r)\bigr)
  +\vol^{2d}\bigl(B_1(r)\bigr)
  +\vol^{2d}\bigl(B_2(r)\bigr)
 \\ \leq
  \vol^{2d}\bigl(B_0(r)\bigr)
  +2\sum_{i=1}^N
  \vol^{2d}\bigl(E_{\mathbb T^d}^{S_iO}(r)\bigr),
\end{multline}
and the number of first-level cells is finite, the preceding estimates
yield $\vol^{2d}\bigl(B(r)\bigr)=O(r^{2d-s+\gamma})$ as
$r\rightarrow 0^+$.
Since $\varphi_{2d}(r)=\vol^{2d}\bigl(B(r)\bigr)$ for
$0<r<r_{\min}$, this proves the
required estimate at small scales. Together with the boundedness argument
at the beginning of the proof, it follows, after enlarging the constant,
that
\[
0\leq\varphi_{2d}(r)\leq Cr^{2d-s+\gamma},
\qquad r>0.
\]
\end{proof}

\subsection{The local two-body asymptotic}
\label{subsec:local-asymptotic}

The local renewal equation established above places the common volume
$f(r)=\vol^{2d}\bigl(E_O^{\Omega^+}(r)\bigr)
=\vol^{2d}\bigl(E_O^{\Omega^-}(r)\bigr)$ precisely in the setting of
Theorem \ref{thm:adapted-renewal}. We record the resulting asymptotic
separately, since it will be used in the subsequent gluing argument.

\begin{corollary}[Local two-body asymptotics]
\label{cor:local-two-body-asymptotics}
Let $F$ be an {admissible local self-similar boundary} of dimension $s$, with
feasible open set $O$.

If $F$ is non-lattice, then there exists a constant $C_F>0$ such that
\begin{equation}
\label{eq:local-two-body-asymptotic}
\vol^{2d}\bigl(E_O^{\Omega^+}(r)\bigr)
=
\vol^{2d}\bigl(E_O^{\Omega^-}(r)\bigr)
=
C_Fr^{2d-s}
+o\bigl(r^{2d-s}\bigr),
\qquad r\rightarrow 0^+.
\end{equation}

If $F$ is lattice and $h>0$ is the maximal lattice span of its IFS,
then there exists a measurable, multiplicatively $e^h$-periodic function
$\Pi_F:(0,\infty)\to(0,\infty)$, bounded above and below by positive
constants, such that
\begin{equation}
\label{eq:local-lattice-two-body-asymptotic}
  \vol^{2d}\bigl(E_O^{\Omega^+}(r)\bigr)
  =
  \vol^{2d}\bigl(E_O^{\Omega^-}(r)\bigr)
  =
  r^{2d-s}\Pi_F(r)
  +o\bigl(r^{2d-s}\bigr),
  \qquad r\rightarrow0^+.
\end{equation}
\end{corollary}

\begin{proof}

After relabeling the IFS if necessary, set
$y_i:=-\log r_i$, $c_i:=r_i^s$, and
\[
  \eta_F:=\sum_{i=1}^N c_i y_i
  =-\sum_{i=1}^N r_i^s\log r_i.
\]
Since $\sum_i r_i^s=1$, the numbers $c_i$ satisfy the normalization in
Theorem \ref{thm:adapted-renewal}. Define
\[
  Z(t):=
  \begin{cases}
    e^{(2d-s)t}f(e^{-t}),&t\geq0,\\
    0,&t<0,
  \end{cases}
  \qquad
  z(t):=
  \begin{cases}
    e^{(2d-s)t}\varphi_{2d}(e^{-t}),&t\geq0,\\
    0,&t<0.
  \end{cases}
\]
After multiplying \eqref{eq:local-renewal} by $e^{(2d-s)t}$ and setting
$r=e^{-t}$, we obtain
\[
  Z(t)=\sum_{i=1}^N c_iZ(t-y_i)+z(t),
  \qquad t\in\mathbb R.
\]
Proposition \ref{thm:local-renewal} shows that $z$ is piecewise
continuous, its discontinuities have no finite accumulation point, and
$|z(t)|\leq Ce^{-\gamma|t|}$ for some $\gamma>0$. Moreover,
$Z(t)\to0$ as $t\to-\infty$. Thus Theorem
\ref{thm:adapted-renewal} applies.

Suppose first that $F$ is non-lattice. Theorem
\ref{thm:adapted-renewal} and the change of variables $r=e^{-t}$ give

\[
  \lim_{r\rightarrow0^+}r^{s-2d}f(r)
  =
  \frac{1}{\eta_F}
  \int_0^1 r^{s-2d-1}\varphi_{2d}(r)\,dr
  =:C_F,
\]
where $\eta_F>0$. The integral is
finite because its integrand is bounded near the origin by
$Cr^{\gamma-1}$. Moreover, $f(r)>0$ for every $r>0$, and
$\varphi_{2d}(r)=f(r)>0$ when $\max_i r_i<r<1$. Hence
$C_F>0$, which proves \eqref{eq:local-two-body-asymptotic}.

Suppose now that $F$ is lattice. Let $h>0$ be its maximal lattice
span. Theorem \ref{thm:adapted-renewal} supplies the
periodic function $p$ in
\eqref{eq:arithmetic-renewal-periodization}. Setting
$\Pi_F(r):=p(-\log r)$, the $h$-periodicity of $p$ gives
$\Pi_F(e^hr)=\Pi_F(r)$.

It remains only to record positivity. The forcing term
$\varphi_{2d}$ is nonnegative by
Proposition \ref{thm:local-renewal}. Moreover, if $0\leq t<h$, then
$e^{-t}>\max_i r_i$, and hence
$\varphi_{2d}(e^{-t})=f(e^{-t})>0$. Thus the term $\ell=0$ in
\eqref{eq:arithmetic-renewal-periodization} is positive on a fundamental
interval. Since $f$ is continuous and positive on
$[e^{-h},1]$, this term is bounded below there by a positive constant.
The decay estimate for $\varphi_{2d}$ makes the periodized series
uniformly convergent on the same interval. Consequently, $p$, and hence
$\Pi_F$, is bounded above and below by positive constants.
Moreover, \eqref{eq:arithmetic-renewal-profile} gives
\[
  r^{s-2d}f(r)-\Pi_F(r)\longrightarrow0,
  \qquad r\rightarrow0^+.
\]
Multiplying by $r^{2d-s}$ proves
\eqref{eq:local-lattice-two-body-asymptotic}.
\end{proof}

\subsection{Passage from local to global}
\label{subsec:local-to-global}

We now pass from the local result established above to its global
counterpart.

Let
We first recall the notation from
Definition \ref{def:admissible-self-similar-boundary}. Let
$\Omega\subset\mathbb T^d$ be a domain whose boundary is an
{admissible self-similar boundary} of dimension $s\in(d-1,d)$. Thus there exist finitely
many {admissible local self-similar boundaries} $F_1,\ldots,F_m$, with
corresponding fixed feasible open sets $O_1,\ldots,O_m$, such that
\begin{equation}
\label{eq:global-boundary-decomposition}
\partial\Omega=\bigcup_{i=1}^mF_i,
\qquad
F_i\subset\overline{O_i},
\qquad
O_i\cap O_j=\varnothing
\quad\text{whenever }i\ne j.
\end{equation}
For each $i$, let
$\mathcal S_i=\{S_{i,1},\ldots,S_{i,N_i}\}$ be the IFS generating
$F_i$. By the definition of an {admissible local self-similar boundary},
$\overline{\dim}_M(\partial O_i)<s$.
After interchanging the two {prescribed sides of $O_i\setminus F_i$} if
necessary, the compatibility condition in the global definition gives
$\Omega_i^+=\Omega\cap O_i$ and
$\Omega_i^-=(\mathbb T^d\setminus\overline{\Omega})\cap O_i$.
We write $E_{O_i}^{\Omega_i^+}(r)$ for the corresponding local two-body set
defined above; this is the orientation compatible with $E_\Omega(r)$. Its
transpose is $E_{O_i}^{\Omega_i^-}(r)$, which has the same product volume. We
retain the notation $E_\Omega(r)$ introduced in the
Introduction for the global two-body set. Throughout this subsection, the
distance condition in these definitions is written as $|x-y|<r$.

The main result of this subsection is the following gluing lemma.

\begin{lemma}[Gluing lemma]
\label{thm:gluing}
Let $\Omega\subset\mathbb T^d$ be a domain whose boundary is an
{admissible self-similar boundary} of dimension $s$, with decomposition
\eqref{eq:global-boundary-decomposition}.

{
If $\partial\Omega$ is non-lattice, then its $s$-dimensional crossing
volume exists and satisfies
\[
  \mathcal V^s(\partial\Omega)
  =\frac{d+1}{\vol(B^{d-1})}
  \sum_{i=1}^m C_{F_i}>0,
\]
where $C_{F_i}$ is the constant in
\eqref{eq:local-two-body-asymptotic}. Equivalently,
\begin{equation}
\label{eq:global-gluing-asymptotic}
\vol^{2d}\bigl(E_\Omega(r)\bigr)
=
\frac{\vol(B^{d-1})}{d+1}\,
\mathcal V^s(\partial\Omega)r^{2d-s}
+
o\bigl(r^{2d-s}\bigr),
\qquad r\rightarrow0^+.
\end{equation}
}

If $\partial\Omega$ is lattice, let $h_i>0$ denote the maximal lattice
span of $F_i$. Suppose that these spans are commensurable, in the
sense that there exists $h_*>0$ such that
\[
  \frac{h_*}{h_i}\in\mathbb N,
  \qquad i=1,\ldots,m.
\]
Then there exist $q>1$ and a measurable, multiplicatively
$q$-periodic function $\Pi_\Omega:(0,\infty)\to(0,\infty)$, bounded
above and below by positive constants, such that
\begin{equation}
\label{eq:global-lattice-gluing-asymptotic}
  \vol^{2d}\bigl(E_\Omega(r)\bigr)
  =
  r^{2d-s}\Pi_\Omega(r)
  +o\bigl(r^{2d-s}\bigr),
  \qquad r\rightarrow0^+.
\end{equation}
\end{lemma}

\begin{proof}
Fix the decomposition and the local sides described above. In particular,
the sets $O_1,\ldots,O_m$ are pairwise disjoint, and
$\Omega_i^+=\Omega\cap O_i$ and
$\Omega_i^-=(\mathbb T^d\setminus\overline{\Omega})\cap O_i$.
For each $i$, let $E_{O_i}^{\Omega_i^+}(r)$ be the local two-body set
defined above. The compatibility of the local and global sides gives
$E_{O_i}^{\Omega_i^+}(r)\subset E_\Omega(r)$. Since the first coordinate of
every pair in $E_{O_i}^{\Omega_i^+}(r)$ lies in $O_i$, the pairwise
disjointness of the sets $O_i$ implies that the local two-body sets are
also pairwise disjoint. Define
\[
B'(r)
:=
E_\Omega(r)\setminus
\bigcup_{i=1}^{m}E_{O_i}^{\Omega_i^+}(r).
\]
We then have the pointwise disjoint decomposition
\begin{equation}
\label{eq:global-local-decomposition}
E_\Omega(r)
=
\left(
\bigsqcup_{i=1}^{m}E_{O_i}^{\Omega_i^+}(r)
\right)
\sqcup B'(r).
\end{equation}
Taking product volumes, we obtain
\begin{equation}
\label{eq:global-local-volume}
\vol^{2d}\bigl(E_\Omega(r)\bigr)
=
\sum_{i=1}^{m}\vol^{2d}\bigl(E_{O_i}^{\Omega_i^+}(r)\bigr)
+
\vol^{2d}\bigl(B'(r)\bigr).
\end{equation}
We first prove that $B'(r)$ is of strictly lower order; this estimate
does not depend on whether the boundary pieces are lattice or non-lattice.

For open $D\subset\mathbb T^d$, we use the notation
$E_{\mathbb T^d}^{D}(r)$ introduced in
Subsection \ref{subsec:local-renewal}. Set
\[
H:=\bigcup_{i=1}^{m}\overline{O_i},
\qquad
A:=\Omega\setminus H,
\qquad
\Sigma:=\bigcup_{i=1}^{m}\partial O_i,
\]
and define
\[
S(r)
:=
E_\Omega(r)\cap
\left[
(\Sigma\times\mathbb T^d)
\cup
(\mathbb T^d\times(\Sigma\cup\partial\Omega))
\right].
\]
Since $\overline{\dim}_{M}(\partial O_i)<s<d$ for every $i$, the set
$\Sigma$ has zero $d$-dimensional measure. The same is true of
$\partial\Omega$, which is a finite union of self-similar sets of
dimension $s<d$. Consequently, $\vol^{2d}\bigl(S(r)\bigr)=0$.

We claim that
\begin{equation}
\label{eq:B-prime-cover}
\begin{aligned}
B'(r)\subset{}&
\bigcup_{i=1}^{m}E_{\mathbb T^d}^{O_i}(r)
\;\cup
\bigcup_{i=1}^{m}
E_{\mathbb T^d}^{\,\mathbb T^d\setminus\overline{O_i}}(r) \\
&{}\cup
E_{\mathbb T^d}^{A}(r)
\cup S(r).
\end{aligned}
\end{equation}
To prove \eqref{eq:B-prime-cover}, let
$(x,y)\in B'(r)\setminus S(r)$.
Since $y\in\mathbb T^d\setminus\Omega$ and
$y\notin\partial\Omega$, we have
$y\in\mathbb T^d\setminus\overline{\Omega}$.

Suppose first that $x\in O_i$ for some $i$. Then $y\notin O_i$;
otherwise the compatibility of the local and global decompositions would
imply $(x,y)\in E_{O_i}^{\Omega_i^+}(r)$, contrary to the definition of
$B'(r)$. It follows that
\[
(x,y)\in E_{\mathbb T^d}^{O_i}(r).
\]

Suppose next that $x\notin\bigcup_iO_i$ but $y\in O_i$ for some $i$.
Since $(x,y)\notin S(r)$, we also have $x\notin\partial O_i$, and
therefore
\[
x\in\mathbb T^d\setminus\overline{O_i}.
\]
Consequently,
\[
(x,y)\in
E_{\mathbb T^d}^{\,\mathbb T^d\setminus\overline{O_i}}(r).
\]

It remains to consider the case in which neither $x$ nor $y$ belongs
to any $O_i$. Since neither coordinate belongs to $\Sigma$, it follows
that $x,y\notin H$. As $x\in\Omega$ and
$y\in\mathbb T^d\setminus\overline{\Omega}$, we obtain
$x\in A$ and $y\notin A$. Thus
\[
(x,y)\in E_{\mathbb T^d}^{A}(r),
\]
which proves \eqref{eq:B-prime-cover}.

We next estimate the sets on the right-hand side of
\eqref{eq:B-prime-cover}. First,
\[
\partial(\mathbb T^d\setminus\overline{O_i})
\subset
\partial O_i
\subset
\Sigma.
\]
Moreover, since each $F_i\subset\overline{O_i}$, we have
$\partial\Omega=\bigcup_{i=1}^{m}F_i\subset H$. Because
$A=\Omega\cap(\mathbb T^d\setminus H)$, the boundary formula for
intersections gives
\[
\begin{aligned}
\partial A
&\subset
\left(
\partial\Omega\cap
\overline{\mathbb T^d\setminus H}
\right)
\cup
\left(
\overline{\Omega}\cap\partial H
\right) \\
&\subset
\partial H
\subset
\bigcup_{i=1}^{m}\partial\overline{O_i}
\subset
\Sigma.
\end{aligned}
\]
By assumption, $\overline{\dim}_M(\partial O_i)<s$ for every $i$.
The preceding inclusions and the finiteness of the decomposition therefore
show that the boundaries of $O_i$,
$\mathbb T^d\setminus\overline{O_i}$, and $A$ all have upper
Minkowski dimension strictly smaller than $s$. Hence we may choose
$q<s$ larger than all these dimensions. By
Remark \ref{rem:minkowski-upper-bound}, for each of these sets $D$,
\[
\vol^{2d}\bigl(E_{\mathbb T^d}^{D}(r)\bigr)
=
O\bigl(r^{2d-q}\bigr),
\qquad r\rightarrow0^+.
\]
Setting $\eta:=s-q>0$, taking product volumes in
\eqref{eq:B-prime-cover}, and using
$\vol^{2d}\bigl(S(r)\bigr)=0$, we obtain
\begin{equation}
\label{eq:B-prime-bound}
\vol^{2d}\bigl(B'(r)\bigr)
=
O\bigl(r^{2d-s+\eta}\bigr)
=
o\bigl(r^{2d-s}\bigr).
\end{equation}

Suppose first that $\partial\Omega$ is non-lattice. For each $i$,
Corollary \ref{cor:local-two-body-asymptotics} gives a constant
$C_{F_i}>0$ such that
\[
  \vol^{2d}\bigl(E_{O_i}^{\Omega_i^+}(r)\bigr)
  =C_{F_i}r^{2d-s}+o\bigl(r^{2d-s}\bigr).
\]
Combining these local asymptotics with
\eqref{eq:global-local-volume} and \eqref{eq:B-prime-bound} proves
{\eqref{eq:global-gluing-asymptotic} and, by
\eqref{eq:crossing-volume-definition},
\[
  \mathcal V^s(\partial\Omega)
  =\frac{d+1}{\vol(B^{d-1})}
  \sum_{i=1}^m C_{F_i}>0.
\]}

Suppose instead that $\partial\Omega$ is lattice and that the maximal
lattice spans of its local pieces are commensurable. Let $h_i$ be the
maximal lattice span of $F_i$, and choose $h_*>0$ as in the statement
of the theorem. By
Corollary \ref{cor:local-two-body-asymptotics},
\[
  \vol^{2d}\bigl(E_{O_i}^{\Omega_i^+}(r)\bigr)
  =r^{2d-s}\Pi_{F_i}(r)+o\bigl(r^{2d-s}\bigr),
\]
where $\Pi_{F_i}$ is multiplicatively $e^{h_i}$-periodic. Hence all
the local profiles are multiplicatively $q:=e^{h_*}$-periodic. The
same two identities and remainder estimate now give
\eqref{eq:global-lattice-gluing-asymptotic}, with
\[
  \Pi_\Omega(r):=\sum_{i=1}^m\Pi_{F_i}(r).
\]
This function is bounded above and below by positive constants and is
multiplicatively $q$-periodic.
\end{proof}

\begin{proof}[Proof of Theorem \ref{thm:admissible-self-similar-kuznecov}]
{
Since $\partial\Omega$ is non-lattice, the first part of
Lemma \ref{thm:gluing} gives
$0<\mathcal V^s(\partial\Omega)<\infty$.
Applying Theorem \ref{thm:crossing-volume-kuznecov}, we obtain
\eqref{eq:intro-two-term-kuznecov}, with
$C_\Omega=C_{d,s}\mathcal V^s(\partial\Omega)$ as in
\eqref{eq:intro-kuznecov-constant}.}
\end{proof}

\section{Kuznecov bounds in the lattice case}
\label{sec:lattice-case}

The preceding sections derived the precise Kuznecov asymptotic in the
non-lattice case. We now turn to the spectral consequences of the lattice
crossing-volume asymptotic obtained in
Section \ref{sec:self-similar-crossing-volumes}. By the renewal theorem
(Theorem \ref{thm:adapted-renewal}), the normalized crossing volume
$r^{s-2d}\vol^{2d}\bigl(E_\Omega(r)\bigr)$ has a log-periodic asymptotic profile as
$r\rightarrow0^+$, rather than converging to a constant.

For the Kuznecov sum, however, such a precise
log-periodic asymptotic cannot in general be obtained. We begin
by recalling the lattice-case theorem stated in the Introduction.

\medskip
\noindent\textbf{Theorem \ref{thm:lattice-kuznecov}
(Kuznecov estimate in the lattice case).}
Let $\Omega\subset\mathbb T^d$ be a domain whose boundary is an
{admissible self-similar boundary} of dimension $s\in(d-1,d)$. Suppose that
$\partial\Omega$ is lattice and that the maximal lattice spans of its
local pieces are commensurable. Then there exist $q>1$, a positive, right-continuous,
multiplicatively $q$-periodic function
$P_\Omega:(0,\infty)\to(0,\infty)$, bounded above and below by positive
constants, and a function $\varepsilon:(0,\infty)\to(0,1)$ satisfying
$\varepsilon(\lambda)\to0$ as $\lambda\to\infty$, such that
\begin{align*}
  \vol(\Omega)
  &-\lambda^{s-d}
    P_\Omega\bigl((1-\varepsilon(\lambda))\lambda\bigr)
    +o(\lambda^{s-d})
  \\
  &\leq N_\Omega(\lambda)
  \leq
  \vol(\Omega)
  -\lambda^{s-d}
    P_\Omega\bigl((1+\varepsilon(\lambda))\lambda\bigr)
  +o(\lambda^{s-d})
\end{align*}
as $\lambda\to\infty$. The two little-oh terms in this estimate may
be different.
\medskip

The reason we obtain only this weaker conclusion lies in the Tauberian
step. The preceding geometric and analytic arguments remain valid, but in
the presence of log-periodic oscillations the corresponding Tauberian
theorem yields convergence only along multiplicative orbits, rather than a
global exact asymptotic. The main purpose of this section is therefore to
state this log-periodic Tauberian theorem and use it, together with
monotonicity, to prove Theorem \ref{thm:lattice-kuznecov}.

{As in Section \ref{sec:heat-tauberian}, set}
\[
  \alpha:=2d-s,
  \qquad
  \beta:=\frac{\alpha-d}{2}=\frac{d-s}{2}.
\]
Recall also that
\begin{equation}
\label{eq:direct-kuznecov-tail}
  R_\Omega(\lambda)
  :=
  \vol(\Omega)-N_\Omega(\lambda)
  =
  \mathcal T_\Omega(\lambda^2).
\end{equation}
Thus $R_\Omega$ is nonnegative, nonincreasing, and right-continuous.

We first record the two conclusions from the preceding sections that carry
over directly to the lattice setting. Suppose that $\partial\Omega$ is
lattice and that the maximal lattice spans of its local pieces are
commensurable. The lattice part of Lemma \ref{thm:gluing} gives a
number $q>1$ and a measurable, multiplicatively $q$-periodic function
$\Pi_\Omega$, bounded above and below by positive constants, such that
\begin{equation}
\label{eq:lattice-crossing-profile}
  {\vol^{2d}\bigl(E_\Omega(r)\bigr)}
  =
  r^\alpha\Pi_\Omega(r)
  +o(r^\alpha),
  \qquad r\rightarrow0^+.
\end{equation}
{
Substituting \eqref{eq:lattice-crossing-profile} into
Lemma \ref{lem:gaussian-average}, making the change of variables
$r=\sqrt t\,u$, and then applying
Proposition \ref{prop:gaussian-to-smoothed-tail}, we obtain
}
\begin{equation}
\label{eq:smoothed-log-periodic-asymptotic}
  H_\Omega(t)
  =
  t^\beta Q_\Omega(t)
  +o(t^\beta),
  \qquad t\rightarrow0^+,
\end{equation}
where
\begin{equation}
\label{eq:smoothed-log-periodic-profile}
  Q_\Omega(t)
  :=
  \frac{(4\pi)^{-d/2}}{2}
  \int_0^\infty
  u^{\alpha+1}e^{-u^2/4}
  \Pi_\Omega(\sqrt t\,u)\,du.
\end{equation}
The function $Q_\Omega$ is continuous, strictly positive, bounded, and
{multiplicatively $q^2$-periodic. These properties follow directly
from \eqref{eq:smoothed-log-periodic-profile}: periodicity follows from
$\Pi_\Omega(qr)=\Pi_\Omega(r)$, positivity and boundedness follow from the
corresponding bounds for $\Pi_\Omega$, and continuity follows by viewing
the integral in logarithmic coordinates as a convolution with an
$L^1$ kernel.}

We next recall two elementary facts that will be used after the Tauberian
step. The first concerns the discontinuities of one-sided continuous
functions; its proof is a standard exercise in real analysis.

\begin{lemma}
\label{lem:right-continuous-discontinuities}
Let $I\subset\mathbb R$ be an interval and let
$f:I\to\mathbb R$ be right-continuous. Then $f$ has at most
countably many discontinuities. In particular, its set of continuity
points is dense in $I$.
\end{lemma}

The second lemma passes from convergence on a dense set to a uniform
comparison on compact intervals.

\begin{lemma}[Uniform comparison for monotone functions]
\label{lem:shifted-monotone-convergence}
Let $H_m,H:(0,\infty)\to\mathbb R$ be nonincreasing functions, and
suppose that $H_m(z)\to H(z)$ at every continuity point $z$ of $H$.
For every compact interval $K\subset(0,\infty)$, there exists a sequence
$\delta_m\to0^+$ such that, for all sufficiently large $m$,
\begin{equation}
\label{eq:shifted-monotone-comparison}
  H(z+\delta_m)-\delta_m
  \leq H_m(z)
  \leq H(z-\delta_m)+\delta_m,
  \qquad z\in K.
\end{equation}
\end{lemma}

\begin{proof}
After enlarging $K$ if necessary, write $K=[c,d]$, where
$0<c<d<\infty$, and choose a decreasing sequence
$\eta_k\to0^+$ such that $\eta_k<c/2$. Since the continuity points
of $H$ are dense, for each $k$ we may choose finitely many such points
\[
  u_{k,0}<u_{k,1}<\cdots<u_{k,J_k}
\]
with
\[
  u_{k,0}<c<d<u_{k,J_k},
  \qquad
  \max_{0\leq j<J_k}(u_{k,j+1}-u_{k,j})<\eta_k.
\]
There exists an integer $M_k$ such that
\begin{equation}
\label{eq:finite-mesh-convergence}
  |H_m(u_{k,j})-H(u_{k,j})|<\eta_k,
  \qquad
  0\leq j\leq J_k,\quad m\geq M_k.
\end{equation}
Increasing the $M_k$ if necessary, we may assume that
$M_1<M_2<\cdots$.

For $M_k\leq m<M_{k+1}$, set $\delta_m:=\eta_k$, and set
$\delta_m:=\eta_1$ when $m<M_1$. Fix $z\in[c,d]$, and choose
$j$ such that $u_{k,j}\leq z\leq u_{k,j+1}$. Since $H_m$ is
nonincreasing, \eqref{eq:finite-mesh-convergence} gives
\[
  H(u_{k,j+1})-\eta_k
  \leq H_m(z)
  \leq H(u_{k,j})+\eta_k.
\]
The mesh condition implies
$z-\eta_k<u_{k,j}\leq z\leq u_{k,j+1}<z+\eta_k$.
The monotonicity of $H$ therefore yields
\[
  H(z+\eta_k)-\eta_k
  \leq H_m(z)
  \leq H(z-\eta_k)+\eta_k.
\]
This is \eqref{eq:shifted-monotone-comparison}. Finally,
$m\to\infty$ implies $k\to\infty$, and hence
$\delta_m\to0^+$.
\end{proof}

We now state the log-periodic Tauberian theorem that relates
\eqref{eq:smoothed-log-periodic-asymptotic} to the sharp spectral tail.
For $r>1$, let $\mathcal P_r$ denote the class of positive, bounded,
right-continuous functions $p:(0,\infty)\to(0,\infty)$ satisfying
$\inf_{1\leq x\leq r}p(x)>0$ and $p(rx)=p(x)$ for every $x>0$.
The following is the case $0<\theta<1$ of
\cite[Corollary 3.1, p. 175]{Kevei2020}, written for finite measures
rather than probability measures.

\begin{theorem}[Log-periodic Tauberian theorem]
\label{thm:log-periodic-tauberian}
Let $\mu$ be a nonzero finite positive Borel measure on
$[0,\infty)$, and set
\[
  \widehat\mu(t):=\int_{[0,\infty)}e^{-tS}\,d\mu(S),
  \qquad
  T_\mu(S):=\mu((S,\infty)).
\]
Let $0<\theta<1$, $r>1$, and let $\ell$ be slowly varying at
infinity. The following assertions are equivalent:

\begin{enumerate}
\item There exists $Q\in\mathcal P_r$ such that
\[
  \mu([0,\infty))-\widehat\mu(t)
  \sim t^\theta\ell(1/t)Q(t),
  \qquad t\rightarrow0^+.
\]

\item There exists $P\in\mathcal P_r$ such that
\begin{equation}
\label{eq:log-periodic-tauberian-tail}
  \lim_{m\to\infty}
  \frac{(r^mz)^\theta}{\ell(r^mz)}
  T_\mu(r^mz)
  =P(z)
\end{equation}
for every $z\in\operatorname{Cont}(P)$.
\end{enumerate}

When these assertions hold, the two profiles are related by
\begin{equation}
\label{eq:log-periodic-profile-transform}
  Q(t)
  =
  \int_0^\infty e^{-u}u^{-\theta}P(u/t)\,du,
  \qquad t>0.
\end{equation}
\end{theorem}

We are now ready to prove Theorem \ref{thm:lattice-kuznecov}.

\begin{proof}[Proof of Theorem \ref{thm:lattice-kuznecov}]
The function $Q_\Omega$ in
\eqref{eq:smoothed-log-periodic-asymptotic} belongs to
$\mathcal P_{q^2}$. Consider the finite spectral measure
\[
  \mu_\Omega
  :=
  \sum_{j=0}^\infty
  \left|\int_\Omega e_j\,dV_g\right|^2\delta_{\lambda_j^2}.
\]
By \eqref{eq:smoothed-tail} and \eqref{eq:sharp-tail},
\[
  \mu_\Omega([0,\infty))-\widehat\mu_\Omega(t)=H_\Omega(t),
  \qquad
  T_{\mu_\Omega}(S)=\mathcal T_\Omega(S).
\]
Since $Q_\Omega$ is bounded away from zero,
\eqref{eq:smoothed-log-periodic-asymptotic} is equivalent to
\[
  H_\Omega(t)\sim t^\beta Q_\Omega(t),
  \qquad t\rightarrow0^+.
\]
Theorem \ref{thm:log-periodic-tauberian}, applied with
$\theta=\beta$, $r=q^2$, and $\ell\equiv1$, gives
$P_T\in\mathcal P_{q^2}$ such that
\[
  \lim_{m\to\infty}
  \bigl((q^2)^mw\bigr)^\beta
  \mathcal T_\Omega\bigl((q^2)^mw\bigr)
  =P_T(w)
\]
at every continuity point $w$ of $P_T$. Define
$P_\Omega(z):=P_T(z^2)$. Then $P_\Omega\in\mathcal P_q$, and
\begin{equation}
\label{eq:orbitwise-log-periodic-limit}
  \lim_{m\to\infty}
  (q^mz)^{2\beta}R_\Omega(q^mz)
  =P_\Omega(z)
\end{equation}
at every continuity point $z$ of $P_\Omega$.

For $m\geq1$ and $z>0$, set
\[
  G_m(z):=q^{2m\beta}R_\Omega(q^mz),
  \qquad
  G(z):=z^{-2\beta}P_\Omega(z).
\]
Each $G_m$ is nonincreasing, and
\eqref{eq:orbitwise-log-periodic-limit} gives
$G_m(z)\to G(z)$ at every continuity point of $G$.
If $u<v$ are continuity points, then
$G_m(u)\geq G_m(v)$; passing to the limit gives
$G(u)\geq G(v)$. Lemma \ref{lem:right-continuous-discontinuities},
the density of the continuity points, and the right-continuity of $G$
extend this inequality to all $0<u<v$. Thus $G$ is nonincreasing.

Lemma \ref{lem:shifted-monotone-convergence}, applied on $[1,q]$,
gives a sequence $\delta_m\to0^+$ such that
\begin{equation}
\label{eq:G-shifted-sandwich}
  G(z+\delta_m)-\delta_m
  \leq G_m(z)
  \leq G(z-\delta_m)+\delta_m,
  \qquad 1\leq z\leq q,
\end{equation}
for all sufficiently large $m$. We may assume that
$0<\delta_m<1/2$.

For every sufficiently large $\lambda$, write
$\lambda=q^mz$, where $m=\lfloor\log_q\lambda\rfloor$ and
$1\leq z<q$, and set
\[
  \varepsilon(\lambda):=\frac{\delta_m}{z}.
\]
Then $\varepsilon(\lambda)\to0$, and
$q^m(z\pm\delta_m)=(1\pm\varepsilon(\lambda))\lambda$.
Multiplying \eqref{eq:G-shifted-sandwich} by $q^{-2m\beta}$ and using
the multiplicative periodicity of $P_\Omega$, we obtain
\begin{align*}
  &\bigl((1+\varepsilon(\lambda))\lambda\bigr)^{-2\beta}
  P_\Omega\bigl((1+\varepsilon(\lambda))\lambda\bigr)
  -q^{-2m\beta}\delta_m
  \\
  &\hspace{3em}\leq R_\Omega(\lambda)
  \leq
  \bigl((1-\varepsilon(\lambda))\lambda\bigr)^{-2\beta}
  P_\Omega\bigl((1-\varepsilon(\lambda))\lambda\bigr)
  +q^{-2m\beta}\delta_m.
\end{align*}
Since $P_\Omega$ is bounded,
\[
  \bigl((1\pm\varepsilon(\lambda))\lambda\bigr)^{-2\beta}
  P_\Omega\bigl((1\pm\varepsilon(\lambda))\lambda\bigr)
  =
  \lambda^{-2\beta}
  P_\Omega\bigl((1\pm\varepsilon(\lambda))\lambda\bigr)
  +o(\lambda^{-2\beta}),
\]
while
\[
  q^{-2m\beta}\delta_m
  =
  \lambda^{-2\beta}z^{2\beta}\delta_m
  =
  o(\lambda^{-2\beta}).
\]
Combining these estimates with the preceding bounds for
$R_\Omega(\lambda)$, we obtain
\begin{align*}
  \lambda^{-2\beta}
  P_\Omega\bigl((1+\varepsilon(\lambda))\lambda\bigr)
  +o(\lambda^{-2\beta})
  &\leq R_\Omega(\lambda)
  \\
  &\leq
  \lambda^{-2\beta}
  P_\Omega\bigl((1-\varepsilon(\lambda))\lambda\bigr)
  +o(\lambda^{-2\beta}),
  \qquad \lambda\to\infty.
\end{align*}
Using \eqref{eq:direct-kuznecov-tail}, the identity
$-2\beta=s-d$, and reversing the inequalities, we conclude that
\begin{align*}
  \vol(\Omega)
  &-\lambda^{s-d}
  P_\Omega\bigl((1-\varepsilon(\lambda))\lambda\bigr)
  +o(\lambda^{s-d})
  \leq N_\Omega(\lambda)
  \\
  &\leq
  \vol(\Omega)
  -\lambda^{s-d}
  P_\Omega\bigl((1+\varepsilon(\lambda))\lambda\bigr)
  +o(\lambda^{s-d}).
\end{align*}
This is \eqref{eq:lattice-kuznecov-estimate}. Finally, setting
$\varepsilon(\lambda)=1/2$ on the remaining bounded interval extends
$\varepsilon$ to $(0,\infty)$.
\end{proof}

\begin{remark}
\label{rem:continuous-periodic-coefficient}
The horizontal displacement in
Theorem \ref{thm:lattice-kuznecov} is needed because
$P_\Omega$ need not be continuous. If $P_\Omega$ is continuous, its
multiplicative periodicity implies uniform continuity on a fundamental
interval, and the values at the displaced arguments may be replaced by
$P_\Omega(\lambda)$. In this case,
\[
  N_\Omega(\lambda)
  =
  \vol(\Omega)
  -\lambda^{s-d}P_\Omega(\lambda)
  +o(\lambda^{s-d}),
  \qquad \lambda\to\infty.
\]
\end{remark}

\bibliographystyle{alpha}
\bibliography{references}

\begin{thebibliography}{RPTWZ25}

\bibitem[BG27]{BerettiGennaioli2024}
Thomas Beretti and Luca Gennaioli.
\newblock Fourier transform of {BV} functions and applications.
\newblock {\em Journal of Functional Analysis}, 292(1):111700, 2027.

\bibitem[BGT87]{BGT1987}
N.~H. Bingham, C.~M. Goldie, and J.~L. Teugels.
\newblock {\em Regular Variation}, volume~27 of {\em Encyclopedia of
  Mathematics and its Applications}.
\newblock Cambridge University Press, Cambridge, 1987.

\bibitem[BGV92]{BGV1992}
Nicole Berline, Ezra Getzler, and Mich{\`e}le Vergne.
\newblock {\em Heat Kernels and Dirac Operators}, volume 298 of {\em
  Grundlehren der Mathematischen Wissenschaften}.
\newblock Springer-Verlag, Berlin, 1992.

\bibitem[D{\'a}v02]{Davila2002}
Juan D{\'a}vila.
\newblock On an open question about functions of bounded variation.
\newblock {\em Calculus of Variations and Partial Differential Equations},
  15(4):519--527, 2002.

\bibitem[Fal14]{Falconer2014}
Kenneth Falconer.
\newblock {\em Fractal Geometry: Mathematical Foundations and Applications}.
\newblock John Wiley \& Sons, Ltd., Chichester, 3 edition, 2014.

\bibitem[FL25]{FrankLarson2025}
Rupert~L. Frank and Simon Larson.
\newblock Riesz means asymptotics for {Dirichlet} and {Neumann} laplacians on
  {Lipschitz} domains.
\newblock {\em Inventiones Mathematicae}, 241(3):999--1079, 2025.

\bibitem[Gri09]{Grigoryan2009}
Alexander Grigor'yan.
\newblock {\em Heat Kernel and Analysis on Manifolds}, volume~47 of {\em AMS/IP
  Studies in Advanced Mathematics}.
\newblock American Mathematical Society; International Press, Providence, RI;
  Boston, MA, 2009.

\bibitem[Hut81]{Hutchinson1981}
John~E. Hutchinson.
\newblock Fractals and self-similarity.
\newblock {\em Indiana University Mathematics Journal}, 30(5):713--747, 1981.

\bibitem[Ivr80]{Ivrii1980}
Victor~Ya. Ivrii.
\newblock Second term of the spectral asymptotic expansion of the
  {Laplace--Beltrami} operator on manifolds with boundary.
\newblock {\em Functional Analysis and Its Applications}, 14(2):98--106, 1980.

\bibitem[Kev20]{Kevei2020}
P{\'e}ter Kevei.
\newblock Regularly log-periodic functions and some applications.
\newblock {\em Probability and Mathematical Statistics}, 40(1):159--182, 2020.

\bibitem[KP10]{KeletiPaquette2010}
Tam{\'a}s Keleti and Elliot Paquette.
\newblock The trouble with von {Koch} curves built from \({n}\)-gons.
\newblock {\em The American Mathematical Monthly}, 117(2):124--137, 2010.

\bibitem[LV96]{LevitinVassiliev1996}
Michael Levitin and Dmitri Vassiliev.
\newblock Spectral asymptotics, renewal theorem, and the {Berry} conjecture for
  a class of fractals.
\newblock {\em Proceedings of the London Mathematical Society}, 72(1):188--214,
  1996.

\bibitem[PX23]{ParkXiao2023}
Hyunchul Park and Yimin Xiao.
\newblock Spectral heat content on a class of fractal sets for subordinate
  killed {Brownian} motions.
\newblock {\em Mathematische Nachrichten}, 296(9):4192--4205, 2023.

\bibitem[RPTWZ25]{RozanovaPierratTeplyaevWinterZaehle2025}
Anna Rozanova-Pierrat, Alexander Teplyaev, Steffen Winter, and Martina
  Z{\"a}hle.
\newblock Fractal curvatures and short-time asymptotics of heat content, 2025.

\bibitem[vdB94]{vanDenBerg1994}
Michiel van~den Berg.
\newblock Heat content and {Brownian} motion for some regions with a fractal
  boundary.
\newblock {\em Probability Theory and Related Fields}, 100(4):439--456, 1994.

\bibitem[vdBdH99]{vanDenBergDenHollander1999}
Michiel van~den Berg and Frank den Hollander.
\newblock Asymptotics for the heat content of a planar region with a fractal
  polygonal boundary.
\newblock {\em Proceedings of the London Mathematical Society}, 78(3):627--661,
  1999.

\bibitem[vdBG15]{vanDenBergGittins2015}
Michiel van~den Berg and Katie Gittins.
\newblock Uniform bounds for the heat content of open sets in euclidean space.
\newblock {\em Differential Geometry and its Applications}, 40:67--85, 2015.

\bibitem[WX23]{WymanXi2023}
Emmett~L. Wyman and Yakun Xi.
\newblock A two term {Kuznecov} sum formula.
\newblock {\em Communications in Mathematical Physics}, 401(2):1127--1162,
  2023.

\bibitem[Xi25]{XiFractalKuznecov}
Yakun Xi.
\newblock {Kuznecov} formulae for fractal measures.
\newblock arXiv preprint arXiv:2512.18379, 2025.

\bibitem[Zel92]{Zelditch1992}
Steven Zelditch.
\newblock {Kuznecov} sum formulae and {Szeg{\H{o}}} limit formulae on
  manifolds.
\newblock {\em Communications in Partial Differential Equations},
  17(1--2):221--260, 1992.

\end{thebibliography}
\end{document}